\documentclass[final,leqno]{siamltex}

\usepackage[intlimits]{mathtools}
\usepackage{amssymb}
\usepackage{color}

\def\norm#1{\|#1\|}

\def\RR{\mathbb{R}}
\def\R{\mathbb{R}}

\def\be{\begin{equation}}
\def\ee{\end{equation}}
\def\bea{\begin{eqnarray}}
\def\eea{\end{eqnarray}}
\def\bean{\begin{eqnarray*}}
\def\eean{\end{eqnarray*}}
\def\div{\mathrm{div}\,}

\makeatletter
\newcommand{\abs}[2][]{\@ifnotempty{#1}{\left}|#2\@ifnotempty{#1}{\right}|}
\makeatother

\newcommand{\Continuous}{\mathcal{C}}

\newcommand{\dint}{\,\mathrm{d}}

\let\div\relax
\DeclareMathOperator{\div}{div}

\DeclareMathOperator*{\esssup}{ess\,sup}

\newcommand{\Frobenius}{\mathrm{F}}

\newcommand{\inner}[2]{\left\langle\,#1\,\middle|\,#2\,\right\rangle}

\newcommand{\Jacobian}{\mathsf{J}}

\newcommand{\Lebesgue}[2][]{\mathsf{L}_{#1}^{#2}}

\newcommand{\minner}[2]{\left\langle\!\left\langle\,#1\,\middle|\,#2\,\right\rangle\!\right\rangle}

\newcommand{\minnerl}[1]{\left\langle\!\left\langle\,#1\,\right|\right.\cdots}

\newcommand{\minnerr}[1]{\cdots\left.\left|\,#1\,\right\rangle\!\right\rangle}

\newcommand{\nablabf}{\mathbf{\nabla}}

\newcommand{\Sobolev}[1]{\mathsf{H}^{#1}}

\newcommand{\SobolevW}[2]{\mathsf{W}^{#1,#2}}

\newcommand{\Sup}{\infty}

\DeclareMathOperator{\tr}{tr}

\newcommand{\vdot}{\,\cdot\,}

\DeclareMathOperator{\vol}{vol}

\makeatletter
\newcommand{\eigennorm}[2][]{\@ifnotempty{#1}{\left}\|#2\@ifnotempty{#1}{\right}\|}
\makeatother

\title{Uniqueness and stability result for Cauchy's equation of motion
       for a certain class of hyperelastic materials\thanks{This 
        work was supported by the German Science Foundation (Deutsche Forschungsgemeinschaft, DFG) under Schu 1978/4-1 and Schu 1978/4-2.}}

\author{A.~W\"ostehoff\thanks{Helmut Schmidt University, Department of Mechanical Engineering,
        Holstenhofweg 85, 22043 Hamburg, Germany ({\tt arne.woestehoff@hsu-hh.de}).}
        \and T.~Schuster\thanks{Saarland University, Department of Mathematics, Campus,
        66123 Saarbr\"ucken, Germany ({\tt thomas.schuster@num.uni-sb.de}), corresponding author.}}

\begin{document}

\maketitle

\begin{abstract}
We consider Cauchy's equation of motion for hyperelastic materials. The solution of this nonlinear initial-boundary value problem is the vector field which discribes the displacement which a particle of this material perceives when exposed to stress and external forces. This equation is of greatest relevance when investigating the behaviour of elastic, anisotropic composites and for the detection of defects in such materials from boundary measurements. Thus results on unique solvability and continuous dependence from the initial values are of large interest in materials research and structural health monitoring. In this article we present such a result, provided that reasonable smoothness assumptions for the displacement field and the boundary of the domain are satisfied, for a certain class of hyperelastic materials where the first Piola-Kirchhoff tensor is written as a conic combination of finitely many given tensors.
\end{abstract}

\begin{keywords} 
Cauchy's equation of motion, hyperelastic materials, uniqueness and stability, Cordes condition, Gronwall's lemma
\end{keywords}

\begin{AMS}
35A01, 35A02, 35L20, 35L70, 74B20
\end{AMS}

\pagestyle{myheadings}
\thispagestyle{plain}
\markboth{A.~W\"OSTEHOFF AND T.~SCHUSTER}{CAUCHY'S EQUATION OF MOTION FOR HYPERELASTIC MATERIALS}

\section{Introduction}

Cauchy's equation of motion follows from conservation of mass and momentum and reads as
\be\label{eq-Cauchy}
\rho (x) \ddot u(t,x) = \div P(t,x) + \rho(x) f(t,x),\qquad t\geq 0,\; x\in \Omega
\ee
where $\Omega\subset \RR^3$ is a bounded domain, $\rho (x)$ denotes the mass density, $u(t,x)$ is the vector of particle displacement, $P(t,x)$ is the first Piola-Kirchhoff stress tensor and $f(t,x)$ an external body force. It describes the discplacement that a particle in position $x\in\Omega$ at time $t$ perceives under stress $P$ and external force $f$. If we specifically investigate the behavior of elastic materials we furthermore need a constitutive law which states a connection between the stress tensor $P$ and the position $x$ as well as the deformation gradient $\Jacobian u$, i.e.
\be\label{eq-CL}   P(t,x) = \hat{P}(x,\Jacobian u (t,x)).   \ee
Actually the stress-strain law \eqref{eq-CL} characterizes elastic materials. A special class of such materials are \emph{hyperelastic} materials, where the constitutive function $\hat{P}$ can be expressed as a derivative of a stored energy function $C$,
\be\label{eq-hyper} \hat{P}(x,Y) = \nabla_Y C (x,Y),\qquad Y\in\RR^{3\times 3},\; \mathrm{det}\,Y >0.  \ee
Here, the derivative $\nabla_Y$ is to be understood componentwise. The class of hyperelastic materials comprehends isotropic materials, Mooney-Rivlin materials, neo-Hookean materials and even elastic fluids. Combining \eqref{eq-Cauchy}, \eqref{eq-CL} and \eqref{eq-hyper} yields the equation of motion for hyperelastic materials
\be\label{eq-motion-hyper} \rho (x) \ddot u(t,x) - \div \nabla_Y C\big(x,\Jacobian u (t,x)\big) = \rho (x) f(t,x).  \ee
For detailled derivations of Cauchy's equation of motion and introductions to elastic and hyperelastic materials we refer to the standard textbooks \cite{CIARLET:88,Holzapfel200003,MARSDEN;HUGHES:83} to name only a few.

Since equation \eqref{eq-motion-hyper} models the behavior of hyperelastic materials, this equation has many applications ranging from engineering to biomedical research. E.g. composite materials like carbon-fibre reinforced epoxy are of growing interest in aircraft construction or wind power stations and thus has a deep, economic impact. Dveloping autonomous structural helath monitoring (SHM) systems for such materials is a current and vivid research field to which not only engineers but also mathematicians and computer scientists contribute. Understanding the behavior of composites and developing numerical solvers for the inverse problems which arise in SHM demand for a deep analysis of (\ref{eq-motion-hyper}) equipped with appropriate initial- and boundary values, see also \cite{Giurgiutiu2008}. Existence- and uniqueness results for special cases, especially for the linearized Cauchy equation, can be found in standard references on systems of hyperbolic equations such as \cite{HOERMANDER:10,JOHN:82,Taylor199606,Wloka1982}. In \cite{Hughes1977} the authors deal with existence and uniqueness of a global solution in nonlinear elasticity and they further prove continuous dependence of the solution from initial values. The existence of weak solutions of the linearized version of \eqref{eq-motion-hyper} can also be proven by means of evolution equations, see \cite{lions1972non}. Of course this list is by far not complete. We prove a novel existence- and uniqueness result where it is important to know how the arising constants of the stability estimates depend on the underlying differential operator. This result, which is the main result of the entire article and stated in Theorem \ref{MainTheorem}, relies on a specific class of constitutive functions which are assumed to be conic combinations of fintely many, given tensors, i.e. we suppose that
\[  \hat{P}(x,Y) = \partial_Y C(x,Y) = \sum_{K=1}^N \alpha_K \div \nabla_Y C_K (x,Y),  \]
where $C_K$ and $\alpha_K\geq 0$, $K=1,\ldots,N$ are given. This setting is inspired by the article of Kaltenbacher and Lorenzi \cite{Kaltenbacher200711}. There the authors also assume such a conic combination but their results do hold for scalar displacements, contant mass density and homogeneous engery functions $C(x,Y)=C(Y)$ only, whereas our results are valid for systems of equations in \emph{any} dimension and spatially variable functions $\rho (x)$ and $C(x,Y)$. This is why we do not only consider the three-dimesnional case, even if that case might be the most prominent case in view of applications, but formulate our setting for arbitrary domains $\Omega\subset\RR^d$ with sufficiently smooth boundary and displacements in $\RR^n$. Equipped with appropriate initial- and homogeneous Dirichlet boundary values this gives the system
\begin{equation}\label{PDE}
  \rho(x)\ddot u(t,x) - \sum_{K=1}^{N}\alpha_K\div\nablabf_Y C_K\bigl(x,\Jacobian u(t,x)\bigr) = \rho(x)f(t,x)
\end{equation}
for $t\in [0,T]$ and $x\in \Omega\subset \RR^d$ along with the boundary conditions
\begin{equation}\label{BoundaryValues}
  u(t,x) = 0,\qquad t\in [0,T],\; x\in \partial\Omega
\end{equation}
and given initial values
\begin{equation}\label{InitialValues}
  \begin{split}
    u(0,\vdot)       &= u_0\in\Sobolev{2}(\Omega,R^n)\text{,}\\
	\dot{u}(0,\vdot) &= u_1\in\Sobolev{1}(\Omega,R^n)\text{.}
	\end{split}
\end{equation}
We will prove existence, uniqueness and continuous dependence from the given initial-boundary values, if $\Omega$ has a $\mathcal{C}^2$-boundary and the solution $u$ as well as the given functions $C_K$ satisfy boundedness estimates for derivatives up to the order $3$ and $4$, respectively. The assertions are stated in Theorem \ref{MainTheorem}. The crucial difficulty of the proof is to show that the constants involved to the stability estimates are uniformly bounded with respect to the coefficients $\alpha_K$.\\
The proof is performed in several steps. First we need a generalization of the Cordes condition. To this end we extend a result stated in \cite{Maugeri200004} (Section 3). The next three main steps of the proof are derivations of upper bounds of the solutions and their derivatives corresponding two different sets of initial values $(u_0,u_1)$, $(\tilde{u}_0,\tilde{u}_1)$ coefficients $\alpha_K$, $\tilde{\alpha}_K$ and forces $f$, $\tilde{f}$ which are outlined in Sections 4.1, 4.2 and 4.3. The concluding step of this extensive proof is described in Section 4.4.


\section{Preliminaries and main result}

Throughout the entire article, $\Omega\subset\R^d$ denotes a bounded, open and convex domain with $\Continuous^2$-boundary and $t\in [0,T]$ is a fixed time interval with $T>0$.
Furthermore we suppose that $\rho:\Omega\to(0,\infty)$ is a function satisfying estimates
\[  \rho_{\mathrm{min}} \leq \inf_{x\in\Omega}\rho(x) \leq \sup_{x\in\Omega}\rho(x) \leq \rho_{\mathrm{max}} \]
for constants $0 < \rho_{\mathrm{min}} \leq \rho_{\mathrm{max}} < \infty$. The divergence of a function $f:[0,T]{\times}\Omega\to R^{n\times d}$ is the mapping $\div f: [0,T]{\times}\Omega\to\R^n$ defined by
\begin{equation*}
  \div f(t,x) := \Big(\sum_{j=1}^{d}\frac{\partial}{\partial x_j}f_i(t,x)\Big)_{i=1,\ldots,n}
\end{equation*}
and the Jacobian $\Jacobian u:[0,T]{\times}\Omega\to\R^{n\times d}$ of a function $u:[0,T]{\times}\Omega\to\R^n$ is
\begin{equation*}
  \Jacobian u(t,x) := \Big(\frac{\partial}{\partial x_j} u_i(t,x)\Big)_{i=1,\ldots,n,\;j=1,\ldots,d}\text{.}
\end{equation*}
The derivative with respect to time is always denoted by a dot like $\dot{u}=\partial_t u$, $\ddot{u}=\partial^2_t u$.

By $\SobolevW{2}{2}_{\gamma_0}(\Omega,\R^n)$ we denote the Sobolev space $\SobolevW{2}{2}_{\gamma_0}(\Omega,\R^n) := \Sobolev{2}(\Omega,\R^n)\cap\Sobolev{1}_0(\Omega,\R^n)$ endowed with the norm
\begin{equation*}
  \norm{\vdot}_{\SobolevW{2}{2}_{\gamma_0}(\Omega,\R^n)} := \left(\sum_{k=1}^{n}\norm{u_k}_{\SobolevW{2}{2}_{\gamma_0}(\Omega)}\right)^{1/2}
	                                                       := \left(\sum_{k=1}^{n}\int_{\Omega}\sum_{\ell=1}^{d}\sum_{j=1}^{d}
																												      \bigl(\partial_{ij}u_k(x)\bigr)^2\dint x\right)^{1/2}
\end{equation*}
turning $\SobolevW{2}{2}_{\gamma_0}(\Omega,\R^n)$ into a  Banach space whose norm is equivalent to the $\Sobolev{2}(\Omega,\R^n)$-norm.
Especially, there is a constant $\hat{K} > 0$,
such that $\norm{f}_{\Sobolev{2}(\Omega,\R^n)}\leq\hat{K}\norm{f}_{\SobolevW{2}{2}_{\gamma_0}(\Omega,\R^n)}$ for all $f\in\SobolevW{2}{2}_{\gamma_0}(\Omega,\R^n)$.

Before stating and proving the main result it is necessary to confine the nonlinearity of the PDE-system \eqref{PDE}. To this end we require for every $K$ the existence of constants $\kappa_K^{[0]}$, $\kappa_K^{[1]}$, $\mu_K^{[0]}$ and $\mu_K^{[1]}$, satisfying
\begin{equation}\label{EstimateCK}
  \kappa_K^{[0]}\norm{Y}_{\Frobenius}^2 \leq C_K(x,Y)
	                                      \leq \mu_K^{[0]}\norm{Y}_{\Frobenius}^2
\end{equation}
and
\begin{equation}\label{EstimateDDCK}
  \kappa_K^{[1]}\norm{H}_{\Frobenius}^2 \leq \minner{H}{\nablabf_Y\nablabf_YC_K(x,Y)H}
	                                      \leq \mu_K^{[1]}\norm{H}_{\Frobenius}^2
\end{equation}
for all $H,Y\in\R^{n\times d}$ and almost all $x\in\Omega$. Here, $\minner{A}{B} := \tr (A^TB)$ denotes the inner product of $n{\times}d$-matrices and $\tr(A)$ the trace of $A$; as is known this inner product induces the Frobenius norm $\norm{A}_{\Frobenius} := \sqrt{\minner{A}{A}}$. Furthermore we assume for any $K=1,\ldots,N$ the existence of constants $\mu_K^{[2]},\ldots,\mu_K^{[7]}$, such that the functions $C_K:\Omega\times\R^{n\times d}\to\R^{n\times d}$ and their derivatives are bounded as
\begin{align}
  \norm{\partial_{Y_{pq}}\partial_{Y_{ij}}\partial_{Y_{k\ell}}C_K}_{\Lebesgue{\infty}(\Omega\times\R^{n\times d})}                  &\leq \mu_{K}^{[2]}
	\label{Bound35}\\
	\norm{\partial_{Y_{ab}}\partial_{Y_{pq}}\partial_{Y_{ij}}\partial_{Y_{k\ell}}C_K}_{\Lebesgue{\infty}(\Omega\times\R^{n\times d})} &\leq \mu_{K}^{[3]}
	\label{Bound36}\\
	\norm{\partial_{\ell}\partial_{Y_{k\ell}}C_K}_{\Lebesgue{\infty}(\Omega\times\R^{n\times d})}                                     &\leq \mu_{K}^{[4]}
	\label{Bound37}\\
	\norm{\partial_{Y_{ij}}\partial_{\ell}\partial_{Y_{k\ell}}C_K}_{\Lebesgue{\infty}(\Omega\times\R^{n\times d})}                    &\leq \mu_{K}^{[5]}
	\label{Bound38}\\
	\norm{\partial_{\ell}\partial_{Y_{ij}}\partial_{Y_{k\ell}}C_K}_{\Lebesgue{\infty}(\Omega\times\R^{n\times d})}                    &\leq \mu_{K}^{[6]}
	\label{Bound39}\\
	\norm{\partial_{Y_{pq}}\partial_{\ell}\partial_{Y_{ij}}\partial_{Y_{k\ell}}C_K}_{\Lebesgue{\infty}(\Omega\times\R^{n\times d})}   &\leq \mu_{K}^{[7]}
	\label{Bound310}
\end{align}
for any $p,i,k,a=1,\ldots,n$ and $q,j,\ell,b=1,\ldots,d$. Additionally, let $Y\mapsto C_K(x,Y)$ be three times continuously differentiable for almost all $x\in\Omega$, and let
\begin{equation}\label{Premise311}
  \partial_{Y_{ij}}\partial_\ell\partial_{Y_{k\ell}}C(x,Y) = \partial_\ell\partial_{Y_{ij}}\partial_{Y_{k\ell}}C(x,Y)
\end{equation}
for any $k,i=1,\ldots,n$ and $\ell,j=1,\ldots,d$. E.g. \eqref{Bound35}--\eqref{Premise311} are fulfilled if $C_K\in\Continuous^4(\overline{\Omega}\times\R^{n\times d})$. Finally we assume that the body force $f$ appearing on the right-hand side of \eqref{PDE} is to be an element of $\SobolevW{1}{1}\bigl((0,T),\Lebesgue{2}(\Omega,\R^n)\bigr)$, which is a the set of all $f\in\Lebesgue{1}\bigl((0,T),\Lebesgue{2}(\Omega,R^n)\bigr)$ satisfying $\dot{f}\in\Lebesgue{1}\bigl((0,T),\Lebesgue{2}(\Omega,R^n)\bigr)$ and which is equipped with the norm
\begin{align*}
  \norm{f}_{\SobolevW{1}{1}((0,T),\Lebesgue{2}(\Omega,\R^n))} &:=           \norm{f}_{\Lebesgue{1}((0,T),\Lebesgue{2}(\Omega,R^n))} +
	                                                                            \norm{\dot{f}}_{\Lebesgue{1}((0,T),\Lebesgue{2}(\Omega,R^n))}\\
	                                                            &\phantom{:}= \int_{0}^{T}\left(\int_{\Omega}\abs{f(t,x)}^2\dint x\right)^{1/2} +
															\left(\int_{\Omega}\abs{\dot{f}(t,x)}^2\dint x\right)^{1/2}\dint t\text{.}
\end{align*}

\begin{theorem}\label{MainTheorem}
Let $u$, $\tilde{u}$ be two solutions of problem (\ref{PDE}),(\ref{BoundaryValues}), (\ref{InitialValues}) corresponding to the parameters and initial values $(\alpha,u_0,u_1,f)$, $(\tilde{\alpha},\tilde{u}_0,\tilde{u}_1,\tilde{f})$, respectively. Furthermore assume that
\begin{align}
	  \begin{aligned}
	    \norm{\partial_\ell\partial_ju}_{\Lebesgue{\infty}((0,T),\Lebesgue{2}(\Omega,\R^n))}         &\leq M_0 \qquad &
			\norm{\partial_\ell\partial_j\tilde{u}}_{\Lebesgue{\infty}((0,T),\Lebesgue{2}(\Omega,\R^n))} &\leq M_0\\
			\norm{\partial_\ell\dot{u}}_{\Lebesgue{\infty}((0,T)\times\Omega)}                           &\leq M_1        &
			\norm{\partial_\ell\dot{u}}_{\Lebesgue{\infty}((0,T)\times\Omega)}                           &\leq M_1
		\end{aligned}\label{APriori316}
	\intertext{and}
	  \begin{aligned}
	    \phantom{\norm{\partial_\ell\partial_ju}_{\Lebesgue{\infty}((0,T),\Lebesgue{2}(\Omega,\R^n))}}         &\phantom{\leq M_0 \qquad} &
			\phantom{\norm{\partial_\ell\partial_j\tilde{u}}_{\Lebesgue{\infty}((0,T),\Lebesgue{2}(\Omega,\R^n))}} &\phantom{\leq M_0}\\[-\baselineskip]
		  \norm{\partial_\ell\partial_j\dot{u}_k}_{\Lebesgue{\infty}((0,T)\times\Omega)}                         &\leq M_2  \qquad          &
			\norm{\partial_\ell\partial_j\dot{\tilde{u}}_k}_{\Lebesgue{\infty}((0,T)\times\Omega)}                 &\leq M_2\\
			\norm{\partial_\ell\partial_ju_k}_{\Lebesgue{\infty}((0,T)\times\Omega)}                               &\leq M_3                  &
			\norm{\partial_\ell\partial_j\tilde{u}_k}_{\Lebesgue{\infty}((0,T)\times\Omega)}                       &\leq M_3
	  \end{aligned}\label{APriori317}
\end{align}
hold for any $k = 1,\ldots,n$ and all $\ell,j=1,\ldots,n$. If, in addition, the dimensions $n$ and $d$ satisfy
\begin{equation}\label{Dimensions}
	  \frac{nd-2}{nd-1}\mu < \kappa
		                     < \frac{nd}{nd-1}\mu\text{,}
\end{equation}
where $\kappa := \sum_{K=1}^{N}\alpha_K\kappa_K^{[1]}$ and $\mu := \sum_{K=1}^{N}\alpha_K\mu_K^{[1]}$,
and if there are constants $\kappa(\alpha)$ and $\mu(\alpha)$, so that $\kappa \geq \kappa(\alpha) > 0$ and $\mu \leq \mu(\alpha)$,
then there exist constants $\overline{C}_0$, $\overline{C}_1$, and $\overline{C}_2$ such that the stability estimate
\begin{align}
	  &\mathrel{\phantom{\leq}} \Bigl[\norm{(\dot{u}-\dot{\tilde{u}})(t,\vdot)}_{\Lebesgue{2}(\Omega,R^n)}^2 +
		                            \kappa(\alpha)\norm{(\Jacobian u-\Jacobian\tilde{u})(t,\vdot)}_{\Lebesgue{2}(\Omega,\R^{n\times d})}^2 \mathop{+}\nonumber\\
		&\mathrel{\phantom{\leq}} \mathop{+} \norm{(\ddot{u}-\ddot{\tilde{u}})(t,\vdot)}_{\Lebesgue{2}(\Omega,R^n)}^2 +
		                            \kappa(\alpha)\norm{(\Jacobian\dot{u}-\Jacobian\dot{\tilde{u}})(t,\vdot)}_{\Lebesgue{2}(\Omega,\R^{n\times d})}^2 \mathop{+}\nonumber\\
		&\mathrel{\phantom{\leq}} \mathop{+} \norm{(u-\tilde{u})(t,\vdot)}_{\Sobolev{2}(\Omega,\R^n)}^2\Bigr]^{1/2}\nonumber\\
		&\leq                       \overline{C}_0\left[\mu(\alpha)\norm{u_0-\tilde{u}_0}_{\Sobolev{2}(\Omega,\R^n)}^2 + \norm{u_1-\tilde{u}_1}_{\Sobolev{1}(\Omega,\R^n)}\right]^{1/2} \mathop{+}\nonumber\\
		&\mathrel{\phantom{\leq}} \mathop{+} \overline{C}_1\norm{f-\tilde{f}}_{\SobolevW{1}{1}((0,T),\Lebesgue{2}(\Omega,\R^n))} + \overline{C}_2\norm{\alpha-\tilde{\alpha}}_{\Sup}\label{MainEstimate}
\end{align}
is valid for all $t\in(0,T)$. Thereby the constants $\overline{C}_0$, $\overline{C}_1$, and $\overline{C}_2$ only depend on $T$, $M_0$, $M_1$, $M_2$, $M_3$,
\begin{align}
	  \overline{C}(\alpha) &:= \sum_{K=1}^{N}\alpha_K\mu_K^{[2]}\left(\sum_{K=1}^{N}\alpha_K\kappa_K^{[1]}\right)^{\!-1}\text{,}\label{DefOverlineC}\\
	  \mbox{and} &\\
		\hat{C}(\alpha)      &:= \frac{\hat{K}}{1-\sqrt{1-\varepsilon}}\sum_{K=1}^{N}
		                           \alpha_K\mu_K^{[1]}\left(\sum_{K=1}^{N}\alpha_K\kappa_K^{[1]}\right)^{\!-2},\nonumber
\end{align}
where $\varepsilon$ is a constant whose existence is ensured by inequality~\eqref{Dimensions}.
Moreover, the constants $\overline{C}_0$, $\overline{C}_1$, and $\overline{C}_2$ are uniformly bounded if $(M_0,M_1,M_2,M_3,\overline{C}(\alpha),\hat{C}(\alpha),T)\in \mathcal{M}$ with $\mathcal{M}\subset (0,\infty)^7$ bounded.\\[1ex]
\end{theorem}

The principal techniques to prove this theorem take advantage of a lemma by Gronwall on the first hand and use a generalization of a known result by Maugeri, Palagachev and Softova in \cite{Maugeri200004}, the so-called \emph{Cordes condition}, on the other hand. To this end we conclude this section by stating Gronwall's lemma as we need it in our proof. The generalization of the Cordes condition is subject of section 3.\\

\begin{lemma}[Gronwall]\label{Gronwall}
  Let $\psi\in\Continuous\bigl((0,T),\R\bigr)$ and $b,k\in\Lebesgue{1}\bigl((0,T),\R\bigr)$ be nonnegative functions.
	If $\psi$ satisfies
	\begin{equation*}
	  \psi(\tau) \leq a + \int_{0}^{\tau}b(t)\psi(t)\dint t + \int_{0}^{\tau}k(t)\psi(t)^p\dint t
	\end{equation*}
	for all $\tau\in[0,T]$ with constants $p\in(0,1)$ and $a\geq 0$, then
	\begin{equation*}
	  \psi(\tau) \leq \exp\left(\int_{0}^{\tau}b(t)\dint t\right)\left[a^{1-p}+(1-p)
		  \int_{0}^{\tau}k(t)\exp\left((p-1)\int_{0}^{t}b(\sigma)\dint\sigma\right)\dint t\right]^{1/(1-p)}
	\end{equation*}
	for all $\tau\in[0,T]$.\\
\end{lemma}

A proof of this version can be found for example in \cite{Bainov199205}.


\section{The Cordes condition}

In this section we prove the mentioned generalization of a result accomplished in \cite{Maugeri200004}. More on the Cordes condition can be found in the original articles \cite{Cordes1956,Cordes1961}.

\begin{theorem}\label{theorem:Cordes}
  Let $d\geq 2$ and $a_{ijk\ell}\in\Lebesgue{\infty}(\Omega,\R)$ for $k,i=1,\ldots,n$ and $j,\ell=1,\ldots,d$.
	Additionally, let there be $\lambda,\varepsilon > 0$ with $\varepsilon<1$ in such a way that
	\begin{equation*}
	       \sum_{k=1}^{n}\sum_{\ell=1}^{d}\sum_{i=1}^{n}\sum_{j=1}^{d}a_{k\ell ij}(x)\eta_{k\ell}\eta_{ij}
		\geq \lambda\sum_{k=1}^{n}\sum_{\ell=1}^{d}\abs{\eta_{k\ell}}^2
	\end{equation*}
	for all $(\eta_{k\ell})\in\R^{n\times d}$ and allmost all $x\in\Omega$ as well as
	\begin{equation}\label{Cordes}
	       \sum_{k=1}^{n}\sum_{\ell=1}^{d}\sum_{i=1}^{n}\sum_{j=1}^{d}a_{k\ell ij}^2(x)\left(\sum_{k=1}^{n}\sum_{\ell=1}^{d}a_{k\ell k\ell}(x)\right)^{\!-2}
		\leq \frac{1}{nd-1+\varepsilon}
	\end{equation}
	for allmost all $x\in\Omega$.
	Then the Dirichlet problem
	\begin{equation*}
	    \sum_{k=1}^{n}\sum_{\ell=1}^{d}\sum_{i=1}^{n}\sum_{j=1}^{d}a_{k\ell ij}(x)e_k\partial_{\ell j}u_i(x)
		= f(x),\quad u\in\Sobolev{2}(\Omega,\R^n)\cap\Sobolev{1}_0(\Omega,\R^n)
	\end{equation*}
	with $e_k$ denoting the $k$th standard basis vector in $\R^n$ admits a unique solution $u$ for every $f\in\Lebesgue{2}(\Omega,\R^n)$.
	Moreover, this solution fulfills
	\begin{equation}\label{EstimateDueToCordes}
	  \norm{u}_{\Sobolev{2}(\Omega,\R^n)} \leq C(\alpha)\norm{f}_{\Lebesgue{2}(\Omega,\R^n)}
	\end{equation}
	with
	\begin{equation*}
	  C(\alpha) := \hat{K}\frac{\esssup_{x\in\Omega}\alpha(x)}{1-\sqrt{1-\varepsilon}}
	\end{equation*}
	and
	\begin{equation*}
		\alpha(x) := \sum_{k=1}^{n}\sum_{\ell=1}^{d}a_{k\ell k\ell}(x)\left(\sum_{k=1}^{n}\sum_{\ell=1}^{d}\sum_{i=1}^{n}\sum_{j=1}^{d}
		               a_{k\ell ij}^2(x)\right)^{\!-1}\text{.}
	\end{equation*}
\end{theorem}

\begin{proof}
  To prove Theorem \ref{theorem:Cordes} we follow the lines of the according proof in \cite{Maugeri200004}.
	Let $\mathcal{L}$ be the differential operator
	\begin{equation*}
	  \mathcal{L}u := \sum_{k=1}^{n}\sum_{\ell=1}^{d}\sum_{i=1}^{n}\sum_{j=1}^{d}a_{k\ell ij}(x)e_k\partial_{\ell j}u_i(x)\text{.}
	\end{equation*}
	Due to the premises $\alpha$ is strictly positive,
	since putting $\eta_{k\ell} := \delta_{k_0k}\delta_{\ell_0\ell}$ with $\delta$ being the Kronecker symbol reveals $a_{k_0\ell_0k_0\ell_0} \geq \lambda > 0$.
	Thus, $\mathcal{L}u = f$ is equivalent to $\Delta u = \alpha f + \Delta u - \alpha\mathcal{L}u$.
	The idea is to analyze the operator $T:\SobolevW{2}{2}_{\gamma_0}(\Omega,\R^n)\to\SobolevW{2}{2}_{\gamma_0}(\Omega,\R^n)$ defined by $Tw := U$,
	where $U$ denotes the unique solution of the Poisson problem
	\begin{equation}\label{eq-poisson}
	  \Delta U = \alpha f + \Delta w - \alpha\mathcal{L}w\in\Lebesgue{2}(\Omega,\R^n)\text{,}\quad U\in\SobolevW{2}{2}_{\gamma_0}(\Omega,\R^n)\text{.}
	\end{equation}
	Existence and uniqueness of a solution of \eqref{eq-poisson} can be seen by applying standard results as shown e.\,g. in \cite{Gilbarg197712} or \cite{Ladyzhenskaya196802}
	to the $k$th component
	\begin{equation*}
	    \Delta(U_k)
		= \alpha f_k + \Delta(w_k) - \alpha\sum_{\ell=1}^{d}\sum_{i=1}^{n}\sum_{j=1}^{d}a_{k\ell ij}\partial_{\ell j}w_i\in\Lebesgue{2}(\Omega)\text{.}
	\end{equation*}
	We focus now at the properties of $T$ and want to show that this mapping is a contraction.
	For this purpose, we draw on the famous \emph{Miranda-Talenti estimate}
	\begin{equation}\label{eq-miranda}
	  \int_{\Omega}\sum_{\ell=1}^{d}\sum_{j=1}^{d}\bigl(\partial_{\ell j}v(x)\bigr)^2\dint x \leq \int_{\Omega}\bigl(\Delta v(x)\bigr)^2\dint x\text{.}
	\end{equation}
	A proof of (\ref{eq-miranda}) can also be found in \cite{Maugeri200004}.
	Let $w_1,w_2\in\SobolevW{2}{2}_{\gamma_0}(\Omega,\R^n)$.
	Then using (\ref{eq-miranda}) and the Cauchy-Schwarz inequality yields
	\begin{align*}
	  &\mathrel{\phantom{\leq}} \norm{Tw_1-Tw_2}_{\SobolevW{2}{2}_{\gamma_0}(\Omega,\R^n)}^2
		=                        \sum_{k=1}^{n}\int_{\Omega}\sum_{\ell=1}^{d}\sum_{j=1}^{d}
		                            \bigl[\partial_{\ell j}\bigl(U_{1,k}(x)-U_{2,k}(x)\bigr)\bigr]^2\dint x\\
		&\leq                     \int_{\Omega}\sum_{k=1}^{n}\left\{\bigl[\Delta\bigl(U_1(x)-U_2(x)\bigr)\bigr]_k\right\}^2\dint x
    =                        \norm{\Delta(w_1-w_2)-\alpha\mathcal{L}(w_1-w_2)}_{\Lebesgue{2}(\Omega,\R^n)}^2\\
    &=                        \sum_{k=1}^{n}\int_{\Omega}\abs[auto]{\sum_{\ell=1}^{d}\sum_{i=1}^{n}\sum_{j=1}^{d}
		                            [\delta_{\ell j}\delta_{ki}-\alpha(x)a_{k\ell ij}(x)]\partial_{\ell j}(w_{1,i}(x)-w_{2,i}(x))}^2\dint x\\
		&\leq                     \int_{\Omega}\left[\sum_{k=1}^{n}\sum_{\ell=1}^{d}\sum_{i=1}^{n}\sum_{j=1}^{d}
		                            \bigl(\delta_{\ell j}\delta_{ki}-\alpha(x)a_{k\ell ij}(x)\bigr)^2\right]\!\!\!
		                            \left[\sum_{\ell=1}^{d}\sum_{i=1}^{n}\sum_{j=1}^{d}\left(\partial_{\ell j}
																\bigl(w_{1,i}(x)-w_{2,i}(x)\bigr)\right)^2\right]\!\!\dint x\text{.}
	\end{align*}
	The expression of the first factor of the integrand can be estimated as
	\begin{align*}
	  &\mathrel{\phantom{\leq}} \sum_{k=1}^{n}\sum_{\ell=1}^{d}\sum_{i=1}^{n}\sum_{j=1}^{d}
		                            \bigl(\delta_{\ell j}\delta_{ki}-\alpha(x)a_{k\ell ij}(x)\bigr)^2\\
		&=                        nd - 2\alpha(x)\sum_{k=1}^{n}\sum_{\ell=1}^{d}a_{k\ell k\ell}(x) +
		                            \alpha^2(x)\sum_{k=1}^{n}\sum_{\ell=1}^{d}\sum_{i=1}^{n}\sum_{j=1}^{d}a_{k\ell ij}^2(x)\\
		&=                        nd - \left(\sum_{k=1}^{n}\sum_{\ell=1}^{d}a_{k\ell k\ell}(x)\right)^2
		                            \left(\sum_{k=1}^{n}\sum_{\ell=1}^{d}\sum_{i=1}^{n}\sum_{j=1}^{d}a_{k\ell ij}^2(x)\right)^{\!-1}\\
		&\leq                     nd - (nd-1+\varepsilon)
		 =                        1 - \varepsilon\text{,}
	\end{align*}
	where we made use of the Cordes condition~\eqref{Cordes}.
	We summarize that
	\begin{align*}
	        \norm{Tw_1-Tw_2}_{\SobolevW{2}{2}_{\gamma_0}(\Omega,\R^n)}^2
		&\leq \int_{\Omega}(1-\varepsilon)\sum_{\ell=1}^{d}\sum_{i=1}^{n}\sum_{j=1}^{d}
		        \left(\partial_{\ell j}\bigl(w_{1,i}(x)-w_{2,i}(x)\bigr)\right)^2\dint x\\
		&=    (1-\varepsilon)\norm{w_1-w_2}_{\SobolevW{2}{2}_{\gamma_0}(\Omega,\R^n)}^2
	\end{align*}
	what proves that $T$ in fact is a contraction in $\SobolevW{2}{2}_{\gamma_0}(\Omega,\R^n)$, since $0<\varepsilon<1$. Due to the Banach fixed-point theorem, $T$ has a unique fixed-point,
	i.\,e. there exists a unique $w\in\SobolevW{2}{2}_{\gamma_0}(\Omega,\R^n)$ satisfying $w = Tw = U$.
	The definition of $T$ implies $\Delta w = \alpha f + \Delta w - \alpha\mathcal{L}w$, which is equivalent to $\mathcal{L}w = f$.
	
	It remains to varify \eqref{EstimateDueToCordes}.
	We have already shown that
	\begin{equation*}
	  \norm{U_1-U_2}_{\SobolevW{2}{2}_{\gamma_0}(\Omega,\R^n)}^2 \leq \norm{\Delta(U_1-U_2)}_{\Lebesgue{2}(\Omega,\R^n)}^2
	\end{equation*}
	Setting $w_1:=w$, with $w$ the unique fixed-point of $T$, and $w_2=0$ yielding $U_2 = Tw_2 = T0 = 0$ we infer
	\begin{align*}
	  \norm{w}_{\SobolevW{2}{2}_{\gamma_0}(\Omega,\R^n)} &\leq \norm{\Delta w}_{\Lebesgue{2}(\Omega,\R^n)}
		                                                    \leq \norm{\alpha f}_{\Lebesgue{2}(\Omega,\R^n)} +
																												       \norm{\Delta w-\alpha\mathcal{L}w}_{\Lebesgue{2}(\Omega,\R^n)}\\
																											 &\leq \esssup_{x\in\Omega}\alpha(x)\norm{f}_{\Lebesgue{2}(\Omega,\R^n)} +
																											         \sqrt{1-\varepsilon}\norm{u}_{\SobolevW{2}{2}_{\gamma_0}(\Omega,\R^n)},
	\end{align*}
	where we again used (\ref{eq-miranda}). The assertion finally follows from the equivalence of the norms
	$\norm{\vdot}_{\Sobolev{2}(\Omega,\R^n)}$ and
	$\norm{\vdot}_{\SobolevW{2}{2}_{\gamma_0}(\Omega,\R^n)}$.
\end{proof}


\section{Proof of Theorem \ref{MainTheorem}}

Before we start with the proof of Theorem \ref{MainTheorem} we note that we may replace $\norm{\vdot}_{\Lebesgue{2}(\Omega,R^n)}$ in estimate~\eqref{MainEstimate} by the equivalent, weighted norm $\norm{f}_{\Lebesgue[\rho]{2}(\Omega,R^n)} := \norm{\rho f}_{\Lebesgue{2}(\Omega,R^n)}$ and get
\begin{align*}
	&\mathrel{\phantom{\leq}} \Bigl[\norm{(\dot{u}-\dot{\tilde{u}})(t,\vdot)}_{\Lebesgue[\rho]{2}(\Omega,R^n)}^2 +
		                          \kappa(\alpha)\norm{(\Jacobian u-\Jacobian\tilde{u})(t,\vdot)}_{\Lebesgue{2}(\Omega,\R^{n\times d})}^2 \mathop{+}\\
	&\mathrel{\phantom{\leq}} \mathop{+} \norm{(\ddot{u}-\ddot{\tilde{u}})(t,\vdot)}_{\Lebesgue[\rho]{2}(\Omega,R^n)}^2 +
		                          \kappa(\alpha)\norm{(\Jacobian\dot{u}-\Jacobian\dot{\tilde{u}})(t,\vdot)}_{\Lebesgue{2}(\Omega,\R^{n\times d})}^2 \mathop{+}\\
	&\mathrel{\phantom{\leq}} \mathop{+} \norm{(u-\tilde{u})(t,\vdot)}_{\Sobolev{2}(\Omega,\R^n)}^2\Bigr]^{1/2}\\
	&\leq                       \overline{C}_0\left[\mu(\alpha)\norm{u_0-\tilde{u}_0}_{\Sobolev{2}(\Omega,\R^n)}^2 +
	                            \norm{u_1-\tilde{u}_1}_{\Sobolev{1}(\Omega,\R^n)}\right]^{1/2} \mathop{+}\\
	&\mathrel{\phantom{\leq}} \mathop{+} \overline{C}_1\norm{f-\tilde{f}}_{\SobolevW{1}{1}((0,T),\Lebesgue{2}(\Omega,\R^n))} +
	                            \overline{C}_2\norm{\alpha-\tilde{\alpha}}_{\Sup}\text{.}
\end{align*}

The proof is subdivided in four parts:
\begin{enumerate}
\item We deduce an upper bound for the norm of $v:=u-\tilde{u}$, which depends on $\alpha,\tilde{\alpha}$, $u_0,\tilde{u}_0$, $u_1,\tilde{u}_1$ and $f,\tilde{f}$ (Section 4.1).
\item We show an upper bound for the time-derivative $z:=\dot{v}$ which additionally depends on $v$ (Section 4.2).
\item We prove an upper bound for the $\Sobolev{2}(\Omega,\R^n)$-norm of $v(\tau,\vdot)$ depending on $\alpha,\tilde{\alpha}$, $u_0,\tilde{u}_0$, $u_1,\tilde{u}_1$ and $f,\tilde{f}$ and other norms of derivatives of $v$ (Section 4.3).
\item We summarize the results so far and finish the proof (Section 4.4).
\end{enumerate}


\subsection{An upper bound for $u-\tilde{u}$}

To derive our aim to bound the norm of $v=u-\tilde{u}$ we at first prove some intermediate results. Thereby the key role will play Gronwall's lemma \ref{Gronwall}.\\

\begin{lemma}\label{LemmaUpperBoundu}
  We have
	\begin{align*}
	  &\mathrel{\phantom{\leq}} \norm{\dot{u}(\tau,\vdot)}_{\Lebesgue[\rho]{2}(\Omega,\R^n)}^2 +
		                            2\sum_{K=1}^{N}\alpha_K\kappa_K^{[0]}\norm{\Jacobian u(\tau,\vdot)}_{\Lebesgue{2}(\Omega,\R^{n\times d})}^2\\
		&\leq                     \left[\left(\norm{u_1}_{\Lebesgue[\rho]{2}(\Omega,\R^n)}^2+2\sum_{K=1}^{N}\alpha_K\mu_K^{[0]}\norm{\Jacobian u_0}_{\Lebesgue{2}(\Omega,R^{n\times d})}^2\right)^{1/2} +
		                          \left(\int_{0}^{\tau}\norm{f(t,\vdot)}_{\Lebesgue{2}(\Omega,\R^n)}\right)^{1/2}\right]^2\text{.}
	\end{align*}
\end{lemma}

\begin{proof}
  Multiplying equation (\ref{PDE}) by $2\dot{u}$ and integrating over $\Omega$ gives
	\begin{align}
	  &\mathrel{\phantom{=}} \partial_t\norm{\dot{u}(t,\vdot)}_{\Lebesgue[\rho]{2}(\Omega,\R^n)}^2 -
		                         2\sum_{K=1}^{N}\alpha_K\inner{\dot{u}(t,\vdot)}{\div\nablabf_YC_K\bigl(\vdot,\Jacobian u(t,\vdot)\bigr)}_{\Lebesgue{2}(\Omega,R^n)}\nonumber\\
	  &=                     2\inner{\rho(\vdot)\dot{u}(t,\vdot)}{f(t,\vdot)}_{\Lebesgue{2}(\Omega,R^n)}\text{.}\label{eq-help1}
	\end{align}
	Using the divergence theorem and the chain rule yields
	\begin{align*}
	  &\mathrel{\phantom{=}} \inner{\dot{u}(t,\vdot)}{\div\nablabf_YC_K\bigl(\vdot,\Jacobian u(t,\vdot)\bigr)}_{\Lebesgue{2}(\Omega,R^n)}\\
		&=                     \sum_{k=1}^{n}\int_{\Omega}\dot{u}_k(t,x)\div\left[e_k^T\nablabf_YC_K\bigl(x,\Jacobian u(t,x)\bigr)\right]\dint x\\
		&=                     -\int_{\Omega}\minner{\nablabf_YC_K\bigl(x,\Jacobian u(t,x)\bigr)}{\Jacobian\dot{u}(t,x)}\dint x\\
		&=                     \partial_t\left[-\int_{\Omega}C_K\bigl(x,\Jacobian u(t,x)\bigr)\dint x\right]\text{.}
	\end{align*}
	If we use this reformulation in \eqref{eq-help1}, we see that
	\begin{align*}
	  &\mathrel{\phantom{=}} \partial_t\left\{\norm{\dot{u}(t,\vdot)}_{\Lebesgue[\rho]{2}(\Omega,\R^n)}^2 +
		                         2\sum_{K=1}^{N}\alpha_K\int_{\Omega}C_K\bigl(x,\Jacobian u(t,x)\bigr)\dint x\right\}\\
		&=                     2\inner{\rho(\vdot)\dot{u}(t,\vdot)}{f(t,\vdot)}_{\Lebesgue{2}(\Omega,R^n)}\text{.}
	\end{align*}
	Applying this together with assumption (\ref{EstimateCK}),
	the fundamental theorem of calculus and (\ref{InitialValues}) we obtain
	\begin{align*}
	  &\mathrel{\phantom{=}} \norm{\dot{u}(\tau,\vdot)}_{\Lebesgue[\rho]{2}(\Omega,\R^n)}^2 +
		                         2\sum_{K=1}^{N}\alpha_K\kappa_K^{[0]}\norm{\Jacobian u(\tau,\vdot)}_{\Lebesgue{2}(\Omega,\R^{n\times d})}^2\\
		&\leq                  \norm{\dot{u}(\tau,\vdot)}_{\Lebesgue[\rho]{2}(\Omega,\R^n)}^2 +
		                         2\sum_{K=1}^{N}\alpha_K\int_{\Omega}C_K\bigl(x,\Jacobian u(\tau,x)\bigr)\dint x\\
		&=                     \int_{0}^{\tau}2\inner{\rho(\vdot)\dot{u}(t,\vdot)}{f(t,\vdot)}_{\Lebesgue{2}(\Omega,R^n)}\dint t +
		                         \norm{u_1}_{\Lebesgue[\rho]{2}(\Omega,\R^n)} \mathop{+}\\
		&\mathrel{\phantom{=}} \mathop{+} 2\sum_{K=1}^{N}\alpha_K\int_{\Omega}C_K\bigl(x,\Jacobian u_0(x)\bigr)\dint x\text{.}
	\end{align*}
	The Cauchy-Schwarz inequality and once more assumption (\ref{EstimateCK}) imply
	\begin{align*}
	  &\mathrel{\phantom{=}} \norm{\dot{u}(\tau,\vdot)}_{\Lebesgue[\rho]{2}(\Omega,\R^n)}^2 +
		                         2\sum_{K=1}^{N}\alpha_K\kappa_K^{[0]}\norm{\Jacobian u(\tau,\vdot)}_{\Lebesgue{2}(\Omega,\R^{n\times d})}^2\\
		&\leq                  \norm{u_1}_{\Lebesgue[\rho]{2}(\Omega,\R^n)} +
		                         2\sum_{K=1}^{N}\alpha_K\mu_K^{[0]}\norm{\Jacobian u_0}_{\Lebesgue{2}(\Omega,\R^{n\times d})}^2 +
		                         2\int_{0}^{\tau}\norm{f(t,\vdot)}_{\Lebesgue{2}(\Omega,\R^n)} \mathop{\times}\\
		&\mathrel{\phantom{=}} \mathop{\times} \left(\norm{\dot{u}(t,\vdot)}_{\Lebesgue[\rho]{2}(\Omega,\R^n)}^2 +
		                         2\sum_{K=1}^{N}\alpha_k\mu_K^{[0]}
														 \norm{\Jacobian u(t,\vdot)}_{\Lebesgue{2}(\Omega,\R^{n\times d})}^2\right)^{1/2}\dint t\text{.}
	\end{align*}
	The assertion now follows from Gronwall's lemma setting $b=0$, $p=1/2$, $k(t)=2\|f(t,\cdot)\|_{\Lebesgue{2}(\Omega,\R^n)}$,
	\[   a = \norm{u_1}_{\Lebesgue[\rho]{2}(\Omega,\R^n)} +
		                         2\sum_{K=1}^{N}\alpha_K\mu_K^{[0]}\norm{\Jacobian u_0}_{\Lebesgue{2}(\Omega,\R^{n\times d})}^2  \]
    and
    \[   \psi (\tau) =  \norm{\dot{u}(\tau,\vdot)}_{\Lebesgue[\rho]{2}(\Omega,\R^n)}^2 +
		                         2\sum_{K=1}^{N}\alpha_K\mu_K^{[0]}\norm{\Jacobian u(\tau,\vdot)}_{\Lebesgue{2}(\Omega,\R^{n\times d})}^2\text{.}  \]
\end{proof}

We proceed by proving an upper bound as in Lemma \ref{LemmaUpperBoundu} for the difference of two solutions $v$. To use again Gronwall's lemma we need a corresponding integral inequality which we will prove as a first step.\\

\begin{lemma}\label{LemmaUpperBoundv}
  For $v = u-\tilde{u}$ we have
	\begin{align*}
	  &\mathrel{\phantom{=}} \norm{\dot{v}(\tau,\vdot)}_{\Lebesgue{2}[\rho](\Omega,\R^n)}^2 +
		                         \sum_{K=1}^{N}\alpha_K\kappa_K^{[1]}\norm{\Jacobian v(\tau,\vdot)}_{\Lebesgue{2}(\Omega,\R^{n\times d})}^2\\
		&\leq                  \norm{u_1-\tilde{u}_1}_{\Lebesgue[\rho](\Omega,\R^n)}^2 +
		                         \sum_{K=1}^{N}\alpha_K\mu_K^{[1]}
														 \norm{\Jacobian u_0-\Jacobian\tilde{u}_0}_{\Lebesgue{2}(\Omega,\R^{n\times d})}^2 \mathop{+}\\
		&\mathrel{\phantom{=}} \mathop{+} (nd)^2M_1\sum_{K=1}^{N}\mu_K^{[2]}\alpha_K
		                         \int_{0}^{\tau}\norm{\Jacobian v(t,\vdot)}_{\Lebesgue{2}(\Omega,\R^n)}^2\dint t \mathop{+}\\
		&\mathrel{\phantom{=}} \mathop{+} 2\int_{0}^{\tau}\left\{\norm{\dot{v}(t,\vdot)}_{\Lebesgue[\rho]{2}(\Omega,\R^n)} +
		                         \sum_{K=1}^{N}\alpha_K\kappa_K^{[1]}\norm{\Jacobian v(t,\vdot)}_{\Lebesgue{2}(\Omega,\R^{n\times d})}^2\right\}^{1/2} \mathop{\times}\\
		&\mathrel{\phantom{=}} \mathop{\times} \Bigl\{\norm{(f-\tilde{f})(t,\vdot)}_{\Lebesgue{2}(\Omega,\R^n)} \mathop{+}\\
		&\mathrel{\phantom{=}} \mathop{+} \sum_{K=1}^{N}\abs{\alpha_k-\tilde{\alpha}_K}d\left[\sqrt{n\vol(\Omega)}\rho_{\mathrm{min}}^{-1}\mu_K^{[4]} +
		                         dnM_0\rho_{\mathrm{min}}^{-1}\mu_K^{[1]}\right]\Bigr\}\dint t\text{.}
	\end{align*}
\end{lemma}

\begin{proof}
  Note that $v$ solves the differential equation
	\begin{align}
	  &\mathrel{\phantom{=}} \rho(x)\ddot{v}(t,x) - \sum_{K=1}^{N}\alpha_K\div\left[\nablabf_YC_K\bigl(x,\Jacobian u(t,x)\bigr) -
		                         \nablabf_YC_K\bigl(x,\Jacobian\tilde{u}(t,x)\bigr)\right]\nonumber\\
		&=                     \rho(x)\bigl(f(t,x)-\tilde{f}(t,x)\bigr) +
		                         \sum_{K=1}^{N}(\alpha_K-\tilde{\alpha}_K)\div\nablabf_YC_K\bigl(x,\Jacobian\tilde{u}(t,x)\bigr)\label{PDEv}\text{.}
	\end{align}
	As in the proof of Lemma \ref{LemmaUpperBoundu} we multiply equation \eqref{PDEv} by $2\dot{v}$. Hence we reformulate the product of $2\dot{v}$ with the sum on the left-hand side of equation \eqref{PDEv} applying Gaussian's divergence theorem, the fundamental theorem of calculus and the chain rule and obtain
	\begin{align*}
	  &\mathrel{\phantom{=}} \inner{2\dot{v}(t,\vdot)}{\div\left[\nablabf_YC_K\bigl(\vdot,\Jacobian u(t,\vdot)\bigr) -
		                         \nablabf_YC_K\bigl(\vdot,\Jacobian\tilde{u}(t,\vdot)\bigr)\right]}_{\Lebesgue{2}(\Omega,\R^n)}\\
		&=                     -2\sum_{k=1}^{n}\int_{\Omega}\inner{e_k^T\left[\nablabf_YC_K\bigl(x,\Jacobian u(t,x)\bigr) -
		                         \nablabf_YC_K\bigl(x,\Jacobian\tilde{u}(t,x)\bigr)\right]}{\nabla\dot{v}_k(t,x)}\dint x\\
		&=                     -2\sum_{k=1}^{n}\sum_{\ell=1}^{d}\int_{\Omega}\int_{0}^{1}
		                         \minner{\nablabf_YC_K\bigl(x,(1-s)\Jacobian u(t,x)+s\Jacobian\tilde{u}(t,x)\bigr)}{\Jacobian v(t,x)} 
        \partial_{\ell}\dot{v}_k(t,x)\dint s\dint x\text{.}
	\end{align*}
	A subsequent application of the product and chain rule gives
	\begin{align*}
	  &\mathrel{\phantom{=}} \inner{2\dot{v}(t,\vdot)}{\div\left[\nablabf_YC_K\bigl(\vdot,\Jacobian u(t,\vdot)\bigr) -
		                         \nablabf_YC_K\bigl(\vdot,\Jacobian\tilde{u}(t,\vdot)\bigr)\right]}_{\Lebesgue{2}(\Omega,\R^n)}\\
		&=                     -\sum_{k=1}^{n}\sum_{\ell=1}^{d}\sum_{i=1}^{n}\sum_{j=1}^{d}
		                       \int_{\Omega}\int_{0}^{1} \Big\{ \partial_{Y_{ij}}\partial_{Y_{k\ell}}C_K\bigl(x,(1-s)\Jacobian u(t,x)+s\Jacobian\tilde{u}(t,x)\bigr) \mathop{\times}\\
		&\mathrel{\phantom{=}} \mathop{\times} \partial_t\bigl(\partial_jv_i(t,x)\partial_\ell v_k(t,x)\bigr) \Big\} \dint s\dint x\\
		&=                     -\partial_t\sum_{k=1}^{n}\sum_{\ell=1}^{d}\sum_{i=1}^{n}\sum_{j=1}^{d}
		                         \int_{\Omega}\partial_jv_i(t,x)\partial_\ell v_k(t,x)\dint s\dint x \mathop{\times}\\
		&\mathrel{\phantom{=}} \mathop{\times} \int_{0}^{1}\partial_{Y_{ij}}\partial_{Y_{k\ell}}C_K
		                         \bigl(x,(1-s)\Jacobian u(t,x)+s\Jacobian\tilde{u}(t,x)\bigr)\dint s\dint x \mathop{+}\\
		&\mathrel{\phantom{=}} \mathop{+} \sum_{k=1}^{n}\sum_{\ell=1}^{d}\sum_{i=1}^{n}\sum_{j=1}^{d}\int_{\Omega}\partial_jv_i(t,x)\partial_\ell v_k(t,x) \mathop{\times}\\
		&\mathrel{\phantom{=}} \mathop{\times} \int_{0}^{1}\big\langle\!\big\langle \nablabf_Y\partial_{Y_{ij}}\partial_{Y_{k\ell}}C_K\bigl(x,(1-s)\Jacobian u(t,x)+s\Jacobian\tilde{u}(t,x)\bigr)\big|
        (1-s)\Jacobian\dot{u}(t,x)+s\Jacobian\dot{\tilde{u}}(t,x) \big\rangle\!\big\rangle\dint s\dint x\text{.}
	\end{align*}
	The dot product of the divergence-term on the right-hand side of \eqref{PDEv} with $2\dot{v}$ is easily computed to
	\begin{align*}
	  &\mathrel{\phantom{=}} \inner{2\dot{v}(t,\vdot)}{\div\nabla_YC_K\bigl(\vdot,\Jacobian\tilde{u}(t,\vdot)\bigr)}\dint x\\
		&\hspace*{1.5cm}=                     2\sum_{k=1}^{n}\sum_{\ell=1}^{d}\int_{\Omega}\dot{v}_k(t,x)\partial_\ell\partial_{Y_{k\ell}}
		                         C_K\bigl(x,\Jacobian\tilde{u}(t,x)\bigr)\dint x \mathop{+}\\
		&\hspace*{1.5cm}\mathrel{\phantom{=}} \mathop{+} 2\sum_{k=1}^{n}\sum_{\ell=1}^{d}\int_{\Omega}\dot{v}_k(t,x)
		                         \minner{\nablabf_YC_K\bigl(x,\Jacobian\tilde{u}(t,x)\bigr)}{\partial_\ell\Jacobian\tilde{u}(t,x)}\dint x\text{.}
	\end{align*}
	We summarize that taking the dot product of \eqref{PDEv} with $2\dot{v}$ gives\\
	\vspace{-2.1cm}
	\begin{eqnarray*}
      \partial_t\norm{\dot{v}(t,\vdot)}_{\Lebesgue[\rho]{2}(\Omega,\R^n)}^2 &+&
	 \sum_{K=1}^{N}\alpha_K\partial_t\sum_{k=1}^{n}\sum_{\ell=1}^{d}\sum_{i=1}^{n}\sum_{j=1}^{d}
		                         \int_{\Omega}\partial_jv_i(t,x)\partial_\ell v_k(t,x)\dint s\dint x \mathop{\times}\\		                         
        &&\hspace*{-3.3cm}\mathop{\times} \int_{0}^{1}\partial_{Y_{ij}}\partial_{Y_{k\ell}}
		                         C_K\bigl(x,(1-s)\Jacobian u(t,x)+s\Jacobian\tilde{u}(t,x)\bigr)\dint s\dint x \mathop{+}\\		                         
        &&\hspace*{-3.3cm}\mathop{-} \sum_{K=1}^{N}\alpha_K\sum_{k=1}^{n}\sum_{\ell=1}^{d}\sum_{i=1}^{n}\sum_{j=1}^{d}
		                         \int_{\Omega}\partial_jv_i(t,x)\partial_\ell v_k(t,x) \mathop{\times}\\		                         
        &&\hspace*{-3.3cm}\mathop{\times} \int_{0}^{1}\left\langle\!\left\langle \nablabf_Y\partial_{Y_{ij}}\partial_{Y_{k\ell}}C_K\bigl(x,(1-s)\Jacobian u(t,x)+s\Jacobian\tilde{u}(t,x)\bigr)\big|
	    (1-s)\Jacobian\dot{u}(t,x)+s\Jacobian\dot{\tilde{u}}(t,x) \right\rangle\!\right\rangle \dint s\dint x\\	    
        &&\hspace*{-3.3cm} = 2\inner{\rho\dot{v}(t,\vdot)}{(f-\tilde{f})(t,\vdot)}_{\Lebesgue{2}(\Omega,\R^n)} +\\		
		&&\hspace*{-3.3cm} + 2\sum_{K=1}^{N}(\alpha_K-\tilde{\alpha}_K)
		                         \Biggl[\sum_{k=1}^{n}\sum_{\ell=1}^{d}\int_{\Omega}\dot{v}_k(t,x)\partial_\ell\partial_{Y_{k\ell}}
														 C_K\bigl(x,\Jacobian\tilde{u}(t,x)\bigr)\dint x \mathop{+}\\														 
         &&\hspace*{-3.3cm} \mathop{+} \sum_{k=1}^{n}\sum_{\ell=1}^{d}\int_{\Omega}\dot{v}_k(t,x)
		                         \minner{\nablabf_YC_K\bigl(x,\Jacobian\tilde{u}(t,x)\bigr)}{\partial_\ell\Jacobian\tilde{u}(t,x)}\dint x\Biggr]\text{.}
	\end{eqnarray*}
	We integrate this identity over $[0,\tau]$ to get the important equality
	\begin{align*}
	  &\mathrel{\phantom{=}} \norm{\dot{v}(t,\vdot)}_{\Lebesgue[\rho]{2}(\Omega,\R^n)}^2 +
		                         \sum_{K=1}^{N}\alpha_K\sum_{k=1}^{n}\sum_{\ell=1}^{d}\sum_{i=1}^{n}\sum_{j=1}^{d}
														 \int_{\Omega}\partial_jv_i(\tau,x)\partial_\ell v_k(\tau,x)\dint s\dint x \mathop{\times}\\													 
		&\mathrel{\phantom{=}} \mathop{\times} \int_{0}^{1}\partial_{Y_{ij}}\partial_{Y_{k\ell}}
		                         C_K\bigl(x,(1-s)\Jacobian u(\tau,x)+s\Jacobian\tilde{u}(\tau,x)\bigr)\dint s\dint x \mathop{+}\\		                         
		&=                     \norm{\dot{v}(0,\vdot)}_{\Lebesgue[\rho]{2}(\Omega,\R^n)}^2 +
		                         \sum_{K=1}^{N}\alpha_K\sum_{k=1}^{n}\sum_{\ell=1}^{d}\sum_{i=1}^{n}\sum_{j=1}^{d}
														 \int_{\Omega}\partial_jv_i(0,x)\partial_\ell v_k(0,x)\dint s\dint x \mathop{\times}\\														 
    &\mathrel{\phantom{=}} \mathop{\times} \int_{0}^{1}\partial_{Y_{ij}}\partial_{Y_{k\ell}}
		                         C_K\bigl(x,(1-s)\Jacobian u(0,x)+s\Jacobian\tilde{u}(0,x)\bigr)\dint s\dint x \mathop{+}\\		                         
		&\mathrel{\phantom{=}} \mathop{+} \int_{0}^{\tau}\Biggl\{2\inner{\rho\dot{v}(t,\vdot)}{(f-\tilde{f})(t,\vdot)}_{\Lebesgue{2}(\Omega,\R^n)} \mathop{+}\\	
		&\mathrel{\phantom{=}} \mathop{+} \sum_{K=1}^{N}\alpha_K\sum_{k=1}^{n}\sum_{\ell=1}^{d}\sum_{i=1}^{n}\sum_{j=1}^{d}
		                         \int_{\Omega}\partial_jv_i(t,x)\partial_\ell v_k(t,x) \mathop{\times}\\		                         
		&\mathrel{\phantom{=}} \mathop{\times} \int_{0}^{1}\minnerl{\nablabf_Y\partial_{Y_{ij}}\partial_{Y_{k\ell}}C_K\bigl(x,(1-s)\Jacobian u(t,x)+s\Jacobian\tilde{u}(t,x)\bigr)}\\		
		&\mathrel{\phantom{=}} \minnerr{(1-s)\Jacobian\dot{u}(t,x)+s\Jacobian\dot{\tilde{u}}(t,x)}\dint s\dint x
\end{align*}
    \begin{align*}		
		&\mathrel{\phantom{=}} \mathop{+} 2\sum_{K=1}^{N}(\alpha_K-\tilde{\alpha}_K)
		                         \Biggl[\sum_{k=1}^{n}\sum_{\ell=1}^{d}\int_{\Omega}\dot{v}_k(t,x)\partial_\ell\partial_{Y_{k\ell}}
														 C_K\bigl(x,\Jacobian\tilde{u}(t,x)\bigr)\dint x \mathop{+}\\														 
		&\mathrel{\phantom{=}} \mathop{+} \sum_{k=1}^{n}\sum_{\ell=1}^{d}\int_{\Omega}\dot{v}_k(t,x)
		                         \minner{\nablabf_YC_K\bigl(x,\Jacobian\tilde{u}(t,x)\bigr)}{\partial_\ell\Jacobian\tilde{u}(t,x)}\dint x\Biggr]\Biggr\}\dint t
	\end{align*}
	which we take as basis for the proof of the bound to be verified. From assumption (\ref{EstimateDDCK}), it is easy to see that the left-hand side is bounded from below by
	\begin{equation*}
	  \norm{\dot{v}(\tau,\vdot)}_{\Lebesgue[\rho]{2}(\Omega,\R^n)}^2 +
		\sum_{K=1}^{N}\alpha_K\kappa_K^{[1]}\norm{\Jacobian v(\tau,\vdot)}_{\Lebesgue{2}(\Omega,\R^{n\times d})}^2\text{.}
	\end{equation*}
	The right-hand side demands for deeper investigation. Assumption (\ref{EstimateDDCK}) along with the Cauchy-Schwarz inequality implies that an upper bound is given by
	\begin{align*}
	  &\mathrel{\phantom{=}} \norm{u_1-\tilde{u}_1}_{\Lebesgue[\rho]{2}(\Omega,\R^n)}^2 +
		                        \sum_{K=1}^{N}\alpha_K\mu_K^{[1]}\norm{\Jacobian u_0-\Jacobian\tilde{u}_0}_{\Lebesgue{2}(\Omega,\R^{n\times d})}^2 \mathop{+}\\
		&\mathrel{\phantom{=}} \mathop{+} 2\int_{0}^{\tau}\left\{\norm{\dot{v}(t,\vdot)}_{\Lebesgue[\rho]{2}(\Omega,\R^n)}^2 +
		                         \sum_{K=1}^{N}\alpha_K\kappa_K^{[1]}\norm{\Jacobian v(t,\vdot)}_{\Lebesgue{2}(\Omega,\R^{n\times d})}^2\right\}^{1/2} \mathop{\times}\\
    &\mathrel{\phantom{=}} \mathop{\times} \norm{(f-\tilde{f})(t,\vdot)}_{\Lebesgue{2}(\Omega,\R^n)}\dint t \mathop{+}\\
		&\mathrel{\phantom{=}} \mathop{+} \sum_{K=1}^{N}\alpha_K\int_{0}^{\tau}\sum_{k=1}^{n}\sum_{\ell=1}^{d}\sum_{i=1}^{n}\sum_{j=1}^{d}
		                         \int_{\Omega}\partial_jv_i(t,x)\partial_\ell v_k(t,x) \mathop{\times}\\
		&\mathrel{\phantom{=}} \mathop{\times} \int_{0}^{1}\minnerl{\nablabf_Y\partial_{Y_{ij}}\partial_{Y_{k\ell}}
		                         C_K\bigl(x,(1-s)\Jacobian u(t,x)+s\Jacobian\tilde{u}(t,x)\bigr)}\\
		&\mathrel{\phantom{=}} \minnerr{(1-s)\Jacobian\dot{u}(t,x)+s\Jacobian\dot{\tilde{u}}(t,x)}\dint s\dint x\dint t\\
		&\mathrel{\phantom{=}} \mathop{+} 2\sum_{K=1}^{N}(\alpha_K-\tilde{\alpha}_K)
		                         \int_{0}^{\tau}\Biggl[\sum_{k=1}^{n}\sum_{\ell=1}^{d}\int_{\Omega}\dot{v}_k(t,x)\partial_\ell\partial_{Y_{k\ell}}
														 C_K\bigl(x,\Jacobian\tilde{u}(t,x)\bigr)\dint x \mathop{+}\\
		&\mathrel{\phantom{=}} \mathop{+} \sum_{k=1}^{n}\sum_{\ell=1}^{d}\int_{\Omega}\dot{v}_k(t,x)
		                         \minner{\nablabf_YC_K\bigl(x,\Jacobian\tilde{u}(t,x)\bigr)}{\partial_\ell\Jacobian\tilde{u}(t,x)}\dint x\Biggr]\dint t\text{.}
	\end{align*}
	We estimate the different terms involving sums separately. Because the coefficients $\alpha_K$ are nonnegative, and according to the a\,priori estimates (\ref{APriori316}) and 
    assumption (\ref{Bound35}), we have
	\begin{align*}
	  &\mathrel{\phantom{=}} \sum_{K=1}^{N}\alpha_K\int_{0}^{\tau}\sum_{k=1}^{n}\sum_{\ell=1}^{d}\sum_{i=1}^{n}\sum_{j=1}^{d}
		                         \int_{\Omega}\partial_jv_i(t,x)\partial_\ell v_k(t,x) \mathop{\times}\\
		&\mathrel{\phantom{=}} \mathop{\times} \int_{0}^{1}\minnerl{\nablabf_Y\partial_{Y_{ij}}\partial_{Y_{k\ell}}
		                         C_K\bigl(x,(1-s)\Jacobian u(t,x)+s\Jacobian\tilde{u}(t,x)\bigr)}\\
		&\mathrel{\phantom{=}} \minnerr{(1-s)\Jacobian\dot{u}(t,x)+s\Jacobian\dot{\tilde{u}}(t,x)}\dint s\dint x\dint t\\
	  &\leq ndM_1\sum_{K=1}^{N}\mu_K^{[2]}\alpha_K\int_{0}^{\tau}\sum_{k=1}^{n}\sum_{\ell=1}^{d}\sum_{i=1}^{n}\sum_{j=1}^{d}
		                         \int_{\Omega}\abs{\partial_jv_i(t,x)}\abs{\partial_\ell v_k(t,x)}\dint x\dint t\\
		&=                     ndM_1\sum_{K=1}^{N}\mu_K^{[2]}\alpha_K\int_{0}^{\tau}
		                         \int_{\Omega}\left(\sum_{k=1}^{n}\sum_{\ell=1}^{d}\abs{\partial_\ell v_k(t,x)}\right)^2\dint x\dint t\\
		&\leq                  (nd)^2M_1\sum_{K=1}^{N}\mu_K^{[2]}\alpha_K
		                         \int_{0}^{\tau}\norm{\Jacobian v(t,\vdot)}_{\Lebesgue{2}(\Omega,\R^{n\times d})}^2\dint t\text{,}
	\end{align*}
	where we used the Cauchy-Schwarz inequality.\\
    The Cauchy-Schwarz inequality, assumption (\ref{Bound37}) and H\"older's inequality show that
	\begin{align*}
	  &\mathrel{\phantom{=}} \abs[auto]{\int_{0}^{\tau}\sum_{k=1}^{n}\sum_{\ell=1}^{d}\int_{\Omega}\dot{v}_k(t,x)\partial_\ell\partial_{Y_{k\ell}}C_K\bigl(x,\Jacobian\tilde{u}(t,x)\bigr)\dint x\dint t}\\
		&\leq                  \sqrt{n}d\mu_K^{[4]}\int_{0}^{\tau}\int_{\Omega}\left(\sum_{k=1}^{n}\abs{\dot{v}_k(t,x)}^2\right)^{1/2}\dint x\dint t\\
		&\leq                  \sqrt{n\vol(\Omega)}d\mu_K^{[4]}\int_{0}^{\tau}\left\{\norm{\dot{v}(t,\vdot)}_{\Lebesgue[\rho]{2}(\Omega,\R^n)}^2 +
		                         \sum_{L=1}^{N}\alpha_L\kappa_L^{[1]}\norm{\Jacobian v(t,\vdot)}_{\Lebesgue{2}(\Omega,\R^{n\times d})}^2\right\}^{\!1/2}\!\!\dint t\text{.}
	\end{align*}
	The left-hand side may be estimated using the Cauchy-Schwarz and
	H\"older inequalities as well as assumption (\ref{EstimateDDCK}) as
	\begin{align*}
	  &\mathrel{\phantom{=}} \abs[auto]{\int_{0}^{\tau}\sum_{k=1}^{n}\sum_{\ell=1}^{d}\int_{\Omega}\dot{v}_k(t,x)
		                         \minner{\nablabf C_K\bigl(x,\Jacobian\tilde{u}(t,x)\bigr)}{\partial_\ell\Jacobian\tilde{u}(t,x)}\dint x\dint t}\\
		&\leq                  \sqrt{n}\mu_K^{[1]}\int_{0}^{\tau}\norm{\dot{v}_k(t,\vdot)}_{\Lebesgue{2}(\Omega,\R^n)}
		                         \left\{\int_{\Omega}\left(\sum_{\ell=1}^{d}\sum_{i=1}^{n}\sum_{j=1}^{d}
														 \abs{\partial_\ell\partial_j\tilde{u}_i(t,x)}\right)^2\dint x\right\}^{\!1/2}\!\!\dint t\\
		&\leq                  d^2nM_0\rho_{\mathrm{min}}^{-1}\mu_K^{[1]}\int_{0}^{\tau}\left\{\norm{\dot{v}(t,\vdot)}_{\Lebesgue[\rho]{2}(\Omega,\R^n)} +
		                         \sum_{L=1}^{N}\alpha_L\kappa_L^{[1]}\norm{\Jacobian v(t,\vdot)}_{\Lebesgue{2}(\Omega,\R^{n\times d})}^2\right\}^{\!1/2}\!\!\dint t\text{,}
	\end{align*}
	where in the last step we used the a\,priori estimates (\ref{APriori316}) and the Cauchy-Schwarz inequality.\\
	Putting all these estimates together and rearranging terms a little bit, we finally arrive at
	\begin{align*}
	  &\mathrel{\phantom{=}} \norm{\dot{v}(\tau,\vdot)}_{\Lebesgue{2}[\rho](\Omega,\R^n)}^2 +
		                         \sum_{K=1}^{N}\alpha_K\kappa_K^{[1]}\norm{\Jacobian v(\tau,\vdot)}_{\Lebesgue{2}(\Omega,\R^{n\times d})}^2\\
		&\leq                  \norm{u_1-\tilde{u}_1}_{\Lebesgue[\rho](\Omega,\R^n)}^2 +
		                        \sum_{K=1}^{N}\alpha_K\mu_K^{[1]}\norm{\Jacobian u_0-\Jacobian\tilde{u}_0}_{\Lebesgue{2}(\Omega,\R^{n\times d})}^2 \mathop{+}\\
		&\mathrel{\phantom{=}} \mathop{+} (nd)^2M_1\sum_{K=1}^{N}\mu_K^{[2]}\alpha_K\int_{0}^{\tau}\norm{\Jacobian v(t,\vdot)}_{\Lebesgue{2}(\Omega,\R^n)}^2\dint t \mathop{+}\\
		&\mathrel{\phantom{=}} \mathop{+} 2\int_{0}^{\tau}\left\{\norm{\dot{v}(t,\vdot)}_{\Lebesgue[\rho]{2}(\Omega,\R^n)} +
		                         \sum_{K=1}^{N}\alpha_K\kappa_K^{[1]}\norm{\Jacobian v(t,\vdot)}_{\Lebesgue{2}(\Omega,\R^{n\times d})}^2\right\}^{1/2} \mathop{\times}\\
		&\mathrel{\phantom{=}} \mathop{\times} \Bigl\{\norm{(f-\tilde{f})(t,\vdot)}_{\Lebesgue{2}(\Omega,\R^n)} \mathop{+}\\
		&\mathrel{\phantom{=}} \mathop{+} \sum_{K=1}^{N}\abs{\alpha_k-\tilde{\alpha}_K}d\left[\sqrt{n\vol(\Omega)}\rho_{\mathrm{min}}^{-1}\mu_K^{[4]} +
		                         dnM_0\rho_{\mathrm{min}}^{-1}\mu_K^{[1]}\right]\Bigr\}\dint t\text{,}
	\end{align*}
	which is exactly the assertion of the lemma.
\end{proof}
	
Finally, we apply Gronwalls lemma again to complete this subsection and get the boundedness result for $v$.\\

\begin{theorem}\label{theorem319}
  Let
	\begin{equation*}
	  \overline{C}(\alpha) := \sum_{K=1}^{N}\alpha_K\mu_K^{[2]}\left(\sum_{K=1}^{N}\alpha_K\kappa_K^{[1]}\right)^{\!-1}\text{.}
	\end{equation*}
	Then $\overline{C}(\alpha)$ is positive and uniformly bounded in $\alpha$ and the estimate
	\begin{align*}
	  &\mathrel{\phantom{=}} \norm{\dot{v}(\tau,\vdot)}_{\Lebesgue[\rho](\Omega,\R^n)}^2 +
		                         \sum_{K=1}^{N}\alpha_K\kappa_K^{[1]}\norm{\Jacobian v(\tau,\vdot)}_{\Lebesgue{2}(\Omega,\R^{n\times d})}^2\\
		&\leq                  \Biggl\{\exp\bigl((nd)^2M_1\overline{C}(\alpha)\tau/2\bigr)\Biggl[\norm{u_1-\tilde{u}_1}_{\Lebesgue[\rho]{2}(\Omega,\R^n)}^2 \mathop{+}\\
		&\mathrel{\phantom{=}} \mathop{+} \sum_{K=1}^{N}\alpha_K\mu_K^{[1]}\norm{\Jacobian u_0-\Jacobian\tilde{u}_0}_{\Lebesgue{2}(\Omega,\R^{n\times d})}^2\Biggr]^{1/2} \mathop{+}\\
		&\mathrel{\phantom{=}} \mathop{+} \int_{0}^{\tau}\norm{(f-\tilde{f})(t,\vdot)}_{\Lebesgue{2}(\Omega,\R^n)}\exp\bigl((nd)^2M_1\overline{C}(\alpha)(\tau-t)/2\bigr)\dint t \mathop{+}\\
		&\mathrel{\phantom{=}} \mathop{+} \sum_{K=1}^{N}\abs{\alpha_K-\tilde{\alpha}_K}\rho_{\mathrm{min}}^{-1}\left[\sqrt{n\vol(\Omega)}\mu_K^{[4]}+dnM_0\mu_K^{[1]}\right] \mathop{\times}\\
		&\mathrel{\phantom{=}} \mathop{\times} \frac{2}{n^2dM_1\overline{C}(\alpha)}\left(\exp\bigl((nd)^2M_1\overline{C}(\alpha)\tau/2\bigr)-1\right)\Biggr\}^2
	\end{align*}
	holds true.\\
\end{theorem}

\begin{proof}
Applying some straightforward estimates we derive
\begin{equation*}
	  0 <    \frac{\min\mu_K^{[2]}}{\max\kappa_K^{[1]}}
		  \leq \overline{C}(\alpha)
			\leq \frac{\max\mu_K^{[2]}}{\min\kappa_K^{[1]}}
			<    \infty
	\end{equation*}
and hence that $\overline{C}(\alpha)$ is positive and uniformly bounded in $\alpha$. The assertion now follows from Lemma \ref{LemmaUpperBoundv} and an appropriate application of Gronwall's lemma (Lemma \ref{Gronwall}).
\end{proof}


\subsection{An upper bound for $\dot{u}-\dot{\tilde{u}}$}

In the next step we prove a upper norm bound for $z=\dot{u}-\dot{\tilde{u}}$. To this end we define $w:=\dot{u}$ and at first prove an upper bound for $w$. Differentiating the PDE system \eqref{PDE} with repect to $t$ and taking the identity (\ref{Premise311}) into account we see that $w$ solves the initial-boundary value problem
\begin{eqnarray}
    &\mathrel{\phantom{=}}& \rho(x)\ddot{w}(t,x) - \sum_{k=1}^{n}\sum_{\ell=1}^{d}\sum_{i=1}^{n}\sum_{j=1}^{d}\sum_{K=1}^{N}
	                           e_k\alpha_K\partial_\ell\bigl[\partial_{Y_{ij}}\partial_{Y_{k\ell}}C_K\bigl(x,\Jacobian u(t,x)\bigr)\partial_j w_i(t,x)\bigr]\nonumber\\
	  &&\label{PDEw}\\
	  &=&                     \rho(x)\dot{f}(t,x)\nonumber
\end{eqnarray}
along with homogeneous Dirichlet boundary values,
\begin{equation}\label{BoundaryValuesw}
  w(t,x) = 0,\qquad (t,x)\in[0,T]\times\partial\Omega
\end{equation}
and initial values
\begin{equation}\label{InitialValuesw}
  \begin{aligned}
    w(0,x)       &=  \dot{u}(0,x)
	                =  u_1(x)\text{,}\\
	  \dot{w}(0,x) &=  \rho(x)^{-1}\sum_{K=1}^{N}\alpha_K\div\bigl[\nablabf_YC_K\bigl(x,\Jacobian u(t,x)\bigr)\bigr] + f(0,x)
	                =: u_2(x)
  \end{aligned}
\end{equation}
for all $x\in\Omega$, if only $u$ solves the IBVP \eqref{PDE}--\eqref{InitialValues}.\\

\begin{lemma}\label{theorem320}
  Let $w$ be a solution of the IBVP \eqref{PDEw}--\eqref{InitialValuesw}.
	Then
	\begin{align*}
	  &\mathrel{\phantom{=}} \norm{\dot{w}(\tau,\vdot)}_{\Lebesgue[\rho]{2}(\Omega,\R^n)}^2 +
		                         \sum_{K=1}^{N}\alpha_K\kappa_K^{[1]}\norm{\Jacobian w(\tau,\vdot)}_{\Lebesgue{2}(\Omega,\R^{n\times d})}^2\\
		&\leq                  \Biggl\{\exp\bigl((nd)^2M_1\overline{C}(\alpha)\tau/2\bigr)\left[\norm{u_2}_{\Lebesgue[\rho]{2}(\Omega,\R^n)}^2 +
		                         \sum_{K=1}^{N}\alpha_K\mu_K^{[1]}\norm{\Jacobian u_1}_{\Lebesgue{2}(\Omega,\R^{n\times d})}^2\right]^{1/2} \mathop{+}\\
		&\mathrel{\phantom{=}} \mathop{+} \int_{0}^{\tau}\norm{f(t,\vdot)}_{\Lebesgue{2}(\Omega,\R^n)}
		                         \exp\bigl((nd)^2M_1\overline{C}(\alpha)(\tau-t)/2\bigr)\dint t\Biggr\}^2
	\end{align*}
	with $\overline{C}$ from equation (\ref{DefOverlineC}).
\end{lemma}

\begin{proof}
  An easy calculation shows that
	\begin{align*}
	  &\mathrel{\phantom{=}} \inner{2\dot{w}(t,\vdot)}{\sum_{k=1}^{n}\sum_{\ell=1}^{d}\sum_{i=1}^{n}\sum_{j=1}^{d}\sum_{K=1}^{N}
		                         e_k\alpha_K\partial_\ell\bigl[\partial_{Y_{ij}}\partial_{Y_{k\ell}}
														 C_K\bigl(\vdot,\Jacobian u(t,\vdot)\bigr)\partial_jw_i(t,\vdot)\bigr]}_{\!\!\Lebesgue{2}(\Omega,\R^n)}\\
		&=                     -\partial_t\sum_{k=1}^{n}\sum_{\ell=1}^{d}\sum_{i=1}^{n}\sum_{j=1}^{d}\sum_{K=1}^{N}
		                         \alpha_K\int_{\Omega}\partial_{Y_{ij}}\partial_{Y_{k\ell}}
														 C_K\bigl(x,\Jacobian u(t,x)\bigr)\partial_jw_i(t,x)\dint x \mathop{+}\\
	  &\mathrel{\phantom{=}} \mathop{+} \sum_{k=1}^{n}\sum_{\ell=1}^{d}\sum_{i=1}^{n}\sum_{j=1}^{d}\sum_{K=1}^{N}
		                         \alpha_K\minner{\nablabf_Y\partial_{Y_{ij}}\partial_{Y_{k\ell}}
														 C_K\bigl(x,\Jacobian u(t,x)\bigr)}{\Jacobian\dot{u}(t,x)}\partial_jw_i(t,x)\partial_\ell w_k(t,x)\dint x\text{.}
	\end{align*}
	Hence multiplying equation (\ref{PDEw}) by $2\dot{w}$ yields
	\begin{align*}
	  &\mathrel{\phantom{=}} \partial_t\Biggl\{\norm{\dot{w}(t,\cdot)}_{\Lebesgue[\rho]{2}(\Omega,\R^n)}^2 \mathop{+}\\
		&\mathrel{\phantom{=}} \mathop{+}\sum_{K=1}^{n}\alpha_K\sum_{k=1}^{n}\sum_{\ell=1}^{d}\sum_{i=1}^{n}\sum_{j=1}^{d}
		                         \int_{\Omega}\partial_{Y_{ij}}\partial_{Y_{k\ell}}C_K\bigl(x,\Jacobian u(t,x)\bigr)\partial_jw_i(t,x)\partial_\ell w_k(t,x)\dint x\Biggr\}\\
		&=                     2\inner{\rho\dot{w}(t,\vdot)}{\dot{f}(t,\vdot)}_{\Lebesgue{2}(\Omega,\R^n)} \mathop{+}\\
		&\mathrel{\phantom{=}} \mathop{+} \sum_{K=1}^{N}\alpha_K\sum_{k=1}^{n}\sum_{\ell=1}^{d}\sum_{i=1}^{n}\sum_{j=1}^{d}
		                         \int_{\Omega}\minner{\nablabf_Y\partial_{Y_{ij}}\partial_{Y_{k\ell}}C_K\bigl(x,\Jacobian u(t,x)\bigr)}{\Jacobian\dot{u}(t,x)} \mathop{\times}\\
		&\mathrel{\phantom{=}} \mathop{\times} \partial_jw_i(t,x)\partial_\ell w_k(t,x)\dint x\text{.}												
	\end{align*}
	With the very same techniques we used in the proof of Lemma \ref{LemmaUpperBoundu},
	that is integration on $[0,\tau]$, application of assumptions (\ref{EstimateDDCK}), (\ref{Bound35}) as well as a\,priori estimates (\ref{APriori316}),
	the Cauchy-Schwarz inequality and Gronwall's lemma,
	we accomplish the desired bound.
\end{proof}

So far we have an upper bound for $w=\dot{u}$ but not for the difference $z$ which we now are heading for. To this end we aim to refine the estimate in Lemma \ref{theorem320} by adding $\norm{u(\tau,\vdot)}_{\Sobolev{2}(\Omega,\R^n)}^2$ and subsequently extracting the square root on both parts. This yields the estimate
\begin{align*}
	&\mathrel{\phantom{=}} \Biggl\{\norm{\ddot{u}(\tau,\vdot)}_{\Lebesgue[\rho]{2}(\Omega,\R^n)}^2 +
	                         \sum_{K=1}^{N}\alpha_K\kappa_K^{[1]}\norm{\Jacobian\dot{u}(\tau,\vdot)}_{\Lebesgue{2}(\Omega,\R^{n\times d})}^2 +
	                       \norm{u(\tau,\vdot)}_{\Sobolev{2}(\Omega,\R^n)}^2\Biggr\}^{1/2}\\
	&\leq                  \bigl(1+\hat{C}(\alpha)\bigr)\Biggl\{\exp\bigl((nd)^2M_1\overline{C}(\alpha)\tau/2\bigr) \mathop{\times}\\
	&\mathrel{\phantom{=}} \mathop{\times}\left[\norm{u_2}_{\Lebesgue[\rho]{2}(\Omega,\R^n)}^2 +
	                         \sum_{K=1}^{N}\alpha_K\mu_K^{[1]}\norm{\Jacobian u_1}_{\Lebesgue{2}(\Omega,\R^{n\times d})}^2\right]^{1/2} \mathop{+}\\
	&\mathrel{\phantom{=}} \mathop{+} \int_{0}^{\tau}\norm{f(t,\vdot)}_{\Lebesgue{2}(\Omega,\R^n)}
	                         \exp\bigl((nd)^2M_1\overline{C}(\alpha)(\tau-t)/2\bigr)\dint t\Biggr\} \mathop{+}\\
	&\mathrel{\phantom{=}} \mathop{+} \hat{C}(\alpha)\left[\rho_{\mathrm{max}}\norm{f(\tau,\vdot)}_{\Lebesgue{2}(\Omega,\R^n)} +
	                         d\sqrt{n\vol(\Omega)}\sum_{K=1}^{N}\alpha_K\mu_K^{[4]}\right]\text{,}
\end{align*}
where
\begin{equation*}
  \hat{C}(\alpha) := \frac{\hat{K}}{1-\sqrt{1-\varepsilon}}\sum_{K=1}^{N}\alpha_K\mu_K^{[1]}\left(\sum_{K=1}^{N}\alpha_K\kappa_K^{[1]}\right)^{-1}\text{.}
\end{equation*}
To prove this inequality we only have to show that
\begin{align}
                         \norm{u(\tau,\vdot)}_{\Sobolev{2}(\Omega,\R^n)}
	&\leq                  \hat{C}(\alpha)\Biggl\{\exp\bigl((nd)^2M_1\overline{C}(\alpha)\tau/2\bigr) \mathop{\times}\nonumber\\
	&\mathrel{\phantom{=}} \mathop{\times}\left[\norm{u_2}_{\Lebesgue[\rho]{2}(\Omega,\R^n)}^2 +
	                         \sum_{K=1}^{N}\alpha_K\mu_K^{[1]}\norm{\Jacobian u_1}_{\Lebesgue{2}(\Omega,\R^{n\times d})}^2\right]^{1/2} \mathop{+}\nonumber\\
	&\mathrel{\phantom{=}} \mathop{+} \int_{0}^{\tau}\norm{f(t,\vdot)}_{\Lebesgue{2}(\Omega,\R^n)}
	                         \exp\bigl((nd)^2M_1\overline{C}(\alpha)(\tau-t)/2\bigr)\dint t \mathop{+}\nonumber\\
	&\mathrel{\phantom{=}} \mathop{+} \rho_{\mathrm{max}}\norm{f(\tau,\vdot)}_{\Lebesgue{2}(\Omega,\R^n)} +
	                         d\sqrt{n\vol(\Omega)}\sum_{K=1}^{N}\alpha_K\mu_K^{[4]}\Biggr\}\label{EqFromLemma324}
\end{align}
holds true.
To see this, we apply the chain rule to equation (\ref{PDE}) and get
\begin{equation}\label{PDEforu}
  \begin{split}
    &\mathrel{\phantom{=}} \sum_{k=1}^{n}\sum_{\ell=1}^{d}\sum_{i=1}^{n}\sum_{j=1}^{d}
		                         a_{k\ell ij}\bigl(x,\Jacobian u(t,x)\bigr)e_k\partial_\ell\partial_ju_i(t,x)\\
	  &=                     \rho(x)\ddot{u}(t,x) - \rho(x)f(x) -
		                         \sum_{k=1}^{n}\sum_{\ell=1}^{d}\sum_{K=1}^{N}\alpha_Ke_k\partial_{\ell}\partial_{Y_{k\ell}}C_K\bigl(x,\Jacobian u(t,x)\bigr)
  \end{split}
\end{equation}
with coefficients
\begin{equation*}
  a_{k\ell ij}\bigl(x,\Jacobian u(t,x)\bigr) := \sum_{K=1}^{N}\alpha_K\partial_{Y_{k\ell}}\partial_{Y_{ij}}C_K\bigl(x,\Jacobian u(t,x)\bigr)
\end{equation*}
being bounded in $t$ and measurable in $x$. To verify \eqref{EqFromLemma324} we want to use the Cordes condition (Theorem \ref{theorem:Cordes}). Doing so we first have to check whether the premisees of Theorem \ref{theorem:Cordes} are satisfied. Using a bijection $\varphi:\{1,\ldots,n\}\times\{1,\ldots,d\}\to\{1,\ldots,nd\}$ we define the matrix $B := (b_{pq})$ by $b_{pq} := a_{\varphi^{-1}(p)\varphi^{-1}(q)}$. Obviously, we have 
$$\sum_{k=1}^{n}\sum_{\ell=1}^{d}a_{k\ell k\ell}(x,Y) = \tr\bigl(B(x,y)\bigr).$$
Since $B$ is symmetric because of the smoothness conditions of $C_K$, an easy calculation shows that
\begin{equation*}
  \sum_{k=1}^{n}\sum_{\ell=1}^{d}\sum_{i=1}^{n}\sum_{j=1}^{d}a_{k\ell ij}^2(x,Y) = \sum_{p=1}^{nd}\lambda_p^2\bigl(B(x,Y)\bigr)\text{,}
\end{equation*}
where $\lambda_p$ is the $p$-th eigenvalue.
Hence, we have
\begin{align*}
  &\mathrel{\phantom{=}} \sum_{k=1}^{n}\sum_{\ell=1}^{d}\sum_{i=1}^{n}\sum_{j=1}^{d}a_{k\ell ij}^2(x,Y)\left(\sum_{k=1}^{n}\sum_{\ell=1}^{d}
	                         a_{k\ell k\ell}(x,Y)\right)^{-2}\\
	&=                     \sum_{p=1}^{nd}\lambda_p^2\left(\bigl[\tr\bigl(B(x,Y)\bigr)\bigr]^{-1}B(x,Y)\right)\text{.}
\end{align*}
A consequence of this identity and estimate (\ref{EstimateCK}) is
\begin{equation*}
  \kappa\norm{y}_2^2 \leq \inner{By}{y}
	                   \leq \mu\norm{y}_2^2 \qquad \mbox{for all } y\in\R^{nd}
\end{equation*}
with $\kappa$ and $\mu$ from Theorem \ref{MainTheorem}.
This implies the lower bound
\begin{equation*}
  \bigl[\tr\bigl(B(x,Y)\bigr)\bigr]^{-1}\kappa \geq \frac{\kappa}{nd\mu}
	                                             =:   \lambda
\end{equation*}
for the eigenvalues of $\bigl[\tr\bigl(B(x,Y)\bigr)\bigr]^{-1}B$.
Now, from (\ref{Dimensions}) we get
\begin{equation*}
  \frac{nd-2}{nd(nd-1)} < \lambda
	                      < \frac{1}{nd-1}\text{,}
\end{equation*}
which is equivalent to
\begin{equation*}
  nd(nd-1)\lambda^2-2(nd-1)\lambda+1 < \frac{1}{nd-1}\text{.}
\end{equation*}
On the other hand, the very same reasoning as used in Section 1.2 of \cite{Maugeri200004} gives
\begin{align*}
  &\mathrel{\phantom{=}} \sum_{k=1}^{n}\sum_{\ell=1}^{d}\sum_{i=1}^{n}\sum_{j=1}^{d}a_{k\ell ij}^2(x,Y)
	                         \left(\sum_{k=1}^{n}\sum_{\ell=1}^{d}a_{k\ell k\ell}(x,Y)\right)^{-2}\\
	&\leq                  nd(nd-1)\lambda^2 - 2(nd-1) + 1\text{.}
\end{align*}
A combination of these results shows the validity of inequality (\ref{Cordes}). Thus all conditions of Theorem \ref{theorem:Cordes} are satisfied
and we may apply now this theorem to the partial differential equation (\ref{PDEforu}) getting
\begin{align*}
                         \norm{u(\tau,\vdot)}_{\Sobolev{2}(\Omega,\R^n)}
	&\leq                  C(\alpha)\Biggl(\norm{\ddot{u}(\tau,\vdot)}_{\Lebesgue[\rho]{2}(\Omega,\R^n)} +
	                         \rho_{\mathrm{max}}\norm{f(\tau,\vdot)}_{\Lebesgue{2}(\Omega,\R^n)} \mathop{+}\\
	&\mathrel{\phantom{=}} \mathop{+} \eigennorm[auto]{\sum_{k=1}^{n}\sum_{\ell=1}^{d}\sum_{K=1}^{N}\alpha_Ke_k\partial_\ell\partial_{Y_{k\ell}}
	                         C_K\bigl(x,\Jacobian u(\tau,\vdot)\bigr)}_{\Lebesgue{2}(\Omega,\R^n)}\Biggr)
\end{align*}
with
\begin{equation*}
  C(\alpha) := \frac{\hat{K}}{1-\sqrt{1-\varepsilon}}\esssup_{x\in\Omega}\frac{\displaystyle\sum_{k=1}^{n}\sum_{\ell=1}^{d}\sum_{K=1}^{N}
	               \alpha_K\partial_{Y_{k\ell}}\partial_{Y_{k\ell}}C_K\bigl(x,\Jacobian u(t,x)\bigr)}
	               {\displaystyle\sum_{k=1}^{n}\sum_{\ell=1}^{d}\sum_{i=1}^{n}\sum_{j=1}^{d}\left(\sum_{K=1}^{N}\alpha_K
								 \partial_{Y_{k\ell}}\partial_{Y_{ij}}C_K\bigl(x,\Jacobian u(t,x)\bigr)\right)^2}\text{.}
\end{equation*}
An evident implication of premise (\ref{EstimateCK}) is that $C(\alpha) \leq \hat{C}(\alpha)$.
Furthermore, from (\ref{Bound37}) we deduce
\begin{equation*}
       \eigennorm[auto]{\sum_{k=1}^{n}\sum_{\ell=1}^{d}\sum_{K=1}^{N}\alpha_Ke_k\partial_\ell\partial_{Y_{k\ell}}
	       C_K\bigl(x,\Jacobian u(\tau,\vdot)\bigr)}_{\Lebesgue{2}(\Omega,\R^n)}
	\leq d\sqrt{n\vol(\Omega)}\sum_{K=1}^{N}\alpha_K\mu_K^{[4]}\text{.}
\end{equation*}
Applying Lemma \ref{theorem320} finally completes the proof of estimate (\ref{EqFromLemma324}).\\[1ex]

Recall that our aim in this section is to bound $z=\dot{u}-\dot{\tilde{u}}$. Again let $w$ be a solution to the IBVP \eqref{PDEw}--\eqref{InitialValuesw} and
let $\tilde{w}$ be a solution to the according IBVP with corresponding coefficients $\alpha_1,\ldots,\alpha_N$,
right-hand side $\rho\dot{\tilde{f}}$ and initial values $\tilde{u}_1$ and $\tilde{u}_2$.
Then $z = w-\tilde{w}$ is a solution of the partial differential equation
\begin{equation}\label{PDEz}
  \rho(x)\ddot{z}(t,x) = \rho(x)\bigl(\dot{f}-\dot{\tilde{f}}\bigr)(t,x) + \Sigma_1(t,x) + \Sigma_2(t,x) + \partial z(t,x)\text{,}
\end{equation}
where
\begin{align*}
  \Sigma_1(t,x)   &:=                     \sum_{k=1}^{n}\sum_{\ell=1}^{d}\sum_{i=1}^{n}\sum_{j=1}^{d}\sum_{K=1}^{N}
	                                          e_k\alpha_K\partial_{\ell}\bigl\{\bigl[\partial_{Y_{ij}}\partial_{Y_{k\ell}}
																						C_K\bigl(x,\Jacobian u(t,x)\bigr) \mathop{+}\\
	                &\mathrel{\phantom{:=}} \mathop{-} \partial_{Y_{ij}}\partial_{Y_{k\ell}}
									                          C_K\bigl(x,\Jacobian\tilde{u}(t,x)\bigr)\bigr]\partial_j\tilde{w}_i(t,x)\bigr\}\text{,}\\
	\Sigma_2(t,x)   &:=                     \sum_{k=1}^{n}\sum_{\ell=1}^{d}\sum_{i=1}^{n}\sum_{j=1}^{d}\sum_{K=1}^{N}
	                                          e_k(\alpha_K-\tilde{\alpha}_K)\partial_{\ell}\bigl[\partial_{Y_{ij}}\partial_{Y_{k\ell}}
																						C_K\bigl(x,\Jacobian\tilde{u}(t,x)\bigr)\partial_j\tilde{w}_i(t,x)\bigr]\text{,}
\intertext{and}
  \partial z(t,x) &:=                     \sum_{k=1}^{n}\sum_{\ell=1}^{d}\sum_{i=1}^{n}\sum_{j=1}^{d}\sum_{K=1}^{N}
		                                        e_k\alpha_K\partial_\ell\bigl[\partial_{Y_{ij}}\partial_{Y_{k\ell}}
																				    C_K\bigl(x,\Jacobian u(t,x)\bigr)\partial_jz_i(t,x)\bigr]\text{.}
\end{align*}
Furthermore $z$ obeys the initial data
\begin{equation}\label{InitialValuesz}
  \begin{split}
    z(0,x)       &=                     (u_1-\tilde{u}_1)(x)\text{,}\\
	  \dot{z}(0,x) &=                     \rho(x)^{-1}\sum_{K=1}^{N}\alpha_K\div\bigl[\nablabf_Y
		                                      C_K\bigl(x,\Jacobian u_0(x)\bigr)-\nablabf_YC_K\bigl(x,\Jacobian\tilde{u}_0(x)\bigr)\bigr] \mathop{+}\\
	               &\mathrel{\phantom{=}} \mathop{+} \rho(x)^{-1}\sum_{K=1}^{N}(\alpha_K-\tilde{\alpha}_K)
								                          \div\bigl[\nablabf_YC_K\bigl(x,\Jacobian\tilde{u}_0(x)\bigr)\bigr] + (f-\tilde{f})(0,x)\\
							   &=                     (u_2-\tilde{u}_2)(x)
	\end{split}
\end{equation}
and boundary values
\begin{equation}\label{BoundaryValuesz}
  z(t,x) = 0\text{.}
\end{equation}

\begin{lemma}
  If $z$ is a solution to problem (\ref{PDEz}), (\ref{InitialValuesz}, (\ref{BoundaryValuesz}), then
	\begin{align}
	  &\mathrel{\phantom{=}} \norm{\dot{z}(\tau,\vdot)}_{\Lebesgue[\rho]{2}(\Omega,\R^n)}^2 +
		                         \sum_{K=1}^{N}\alpha_K\kappa_K^{[1]}\norm{\Jacobian z(\tau,\vdot)}_{\Lebesgue{2}(\Omega,\R^{n\times d})}^2\nonumber\\
		&\leq                  \norm{u_2-\tilde{u}_2}_{\Lebesgue[\rho]{2}(\Omega,\R^n)}^2 +
		                         \sum_{K=1}^{N}\alpha_K\mu_K^{[1]}\norm{\Jacobian u_1-\Jacobian\tilde{u}_1}_{\Lebesgue{2}(\Omega,\R^{n\times d})}^2 \mathop{+}\nonumber\\
		&\mathrel{\phantom{=}} \mathop{+} (nd)^2M_1\left(\sum_{K=1}^{N}\alpha_K\kappa_K^{[1]}\right)^{-1}\sum_{K=1}^{N}\alpha_K\mu_K^{[2]} \mathop{\times}\nonumber\\
		&\mathrel{\phantom{=}} \mathop{\times} \int_{0}^{\tau}\left[\norm{\dot{z}(t,\vdot)}_{\Lebesgue[\rho]{2}(\Omega,\R^n)} +
		                         \sum_{K=1}^{N}\alpha_K\kappa_K^{[1]}\norm{\Jacobian z(t,\vdot)}_{\Lebesgue{2}(\Omega,R^{n\times d})}^2\right]\dint t \mathop{+}\nonumber\\
		&\mathrel{\phantom{=}} \mathop{+} 2\int_{0}^{\tau}\Biggl\{\norm{(\dot{f}-\dot{\tilde{f}})(t,\vdot)}_{\Lebesgue{2}(\Omega,\R^n)} \mathop{+}\nonumber\\
		&\mathrel{\phantom{=}} \mathop{+} \norm{\alpha-\tilde{\alpha}}_{\Sup}\rho_{\mathrm{min}}^{-1}
		                         \Biggl[2\sqrt{2\vol(\Omega)}(nd)^2\sum_{K=1}^{N}\left(M_1\mu_K^{[6]}+M_2\mu_K^{[1]}\right) \mathop{+}\nonumber\\
		&\mathrel{\phantom{=}} \mathop{+} 2\sqrt{2}M_0M_1n^{5/2}d^3\sum_{K=1}^{N}\mu_K^{[2]}\Biggr] \mathop{+}\nonumber\\
		&\mathrel{\phantom{=}} 2\sqrt{2}(nd)^{5/2}\rho_{\mathrm{min}}^{-1}\sum_{K=1}^{N}\alpha_K\left(M_1\mu_K^{[7]} + M_2\mu_K^{[2]} +
		                         M_1M_3nd\mu_K^{[3]}\right) \mathop{\times}\nonumber\\
		&\mathrel{\phantom{=}} \mathop{\times} \norm{\Jacobian v(t,\vdot)}_{\Lebesgue{2}(\Omega,\R^n)} +
		                         2\sqrt{2}M_1n^{5/2}d^2\rho_{\mathrm{min}}^{-1}\sum_{K=1}^{N}
														 \alpha_K\mu_K^{[2]}\norm{v(t,\vdot)}_{\Sobolev{2}(\Omega,\R^n)}\Biggr\} \mathop{\times}\nonumber\\
	  &\mathrel{\phantom{=}} \mathop{\times} \left\{\norm{\dot{z}(t,\vdot)}_{\Lebesgue[\rho]{2}(\Omega,\R^n)}^2 +
		                         \sum_{K=1}^{N}\alpha_K\kappa_K^{[1]}\norm{\Jacobian z(t,\vdot)}_{\Lebesgue{2}(\Omega,\R^{n\times d})}^2\right\}^{1/2}\dint t
		\label{EstimateForz}
	\end{align}
	and the initial values $u_2$, $\tilde{u}_2$ in (\ref{InitialValuesz}) satisfy
	\begin{align}
	  &\mathrel{\phantom{=}} \norm{u_2-\tilde{u}_2}_{\Lebesgue{2}(\Omega,\R^n)}\nonumber\\
		&\leq                  \sum_{K=1}^{N}\abs{\alpha_K-\tilde{\alpha}_K}\left(\sqrt{6\vol(\Omega)}n^{1/2}d\mu_K^{[4]} +
		                         \sqrt{6}nd\mu_K^{[1]}\norm{\tilde{u}_0}_{\Sobolev{2}(\Omega,\R^n)}\right) \mathop{+}\nonumber\\
		&\mathrel{\phantom{=}} \mathop{+} 3nd\sum_{K=1}^{N}\alpha_K\left(d^{1/2}\mu_K^{[5]} + \mu_K^{[1]} +
		                         nd^{3/2}M_3\mu_K^{[2]}\right)\norm{u_0-\tilde{u}_0}_{\Sobolev{2}(\Omega,\R^n)} \mathop{+}\nonumber\\
		&\mathrel{\phantom{=}} \mathop{+} \sqrt{3}\rho_{\mathrm{max}}\norm{(f-\tilde{f})(0,\vdot)}_{\Lebesgue{2}(\Omega,\R^n)}\text{.}\label{EstimateForIVz}
	\end{align}
\end{lemma}

\begin{proof}
  We start by multiplying equation (\ref{PDEz}) with $2\dot{z}$.
	Since
	\begin{equation*}
	  \inner{2\dot{z}(t,x)}{\partial z(t,x)}_{\Lebesgue{2}(\Omega,\R^n)} = -\partial_tI_1(t) + I_2(t)
	\end{equation*}
	with
	\begin{align*}
	  I_1(t) &:= \sum_{k=1}^{n}\sum_{\ell=1}^{d}\sum_{i=1}^{n}\sum_{j=1}^{d}\sum_{K=1}^{N}
		           \alpha_K\int_{\Omega}\partial_{Y_{ij}}\partial_{Y_{k\ell}}C_K\bigl(x,\Jacobian u(t,x)\bigr)\partial_jz_i(t,x)\partial_{\ell}z_k(t,x)\dint x
	\intertext{and}
		I_2(t) &:=              \sum_{k=1}^{n}\sum_{\ell=1}^{d}\sum_{i=1}^{n}\sum_{j=1}^{d}\sum_{K=1}^{N}\alpha_K
		                          \int_{\Omega}\minner{\nablabf_Y\partial_{Y_{ij}}\partial_{Y_{k\ell}}
															C_K\bigl(x,\Jacobian u(t,x)\bigr)}{\Jacobian\dot{u}(t,x)} \mathop{\times}\\
		&\mathrel{\phantom{:=}} \mathop{\times} \partial_jz_i(t,x)\partial_{\ell}z_k(t,x)\dint x
	\end{align*}
	as an application of the divergence theorem shows, a subsequent integration over $[0,\tau]$ yields
	\begin{align*}
	  &\mathrel{\phantom{=}} \norm{\dot{z}(\tau,\vdot)}_{\Lebesgue[\rho]{2}(\Omega,\R^n)}^2 + I_1(\tau)\\
		&=                     \norm{\dot{z}(0,\vdot)}_{\Lebesgue[\rho]{2}(\Omega,\R^n)}^2 + I_1(0) +
		                       2\int_{0}^{\tau}\inner{\rho\dot{z}(t,\vdot)}{(\dot{f}-\dot{\tilde{f}})(t,\vdot)}_{\Lebesgue{2}(\Omega,\R^n)}\dint t \mathop{+}\\
		&\mathrel{\phantom{=}} \mathop{+} \int_{0}^{\tau}I_2(t)\dint t + 2\int_{0}^{\tau}
		                         \inner{\dot{z}(t,\vdot)}{\Sigma_1(t,\vdot)}_{\Lebesgue{2}(\Omega,\R^n)}\dint t \mathop{+}\\
		&\mathrel{\phantom{=}} \mathop{+} 2\int_{0}^{\tau}\inner{\dot{z}(t,\vdot)}{\Sigma_2(t,\vdot)}_{\Lebesgue{2}(\Omega,\R^n)}\dint t\text{.}
	\end{align*}
	We apply premise (\ref{EstimateDDCK}) to $I_1$ and use the Cauchy-Schwarz inequality, which gives us
	\begin{align}
	  &\mathrel{\phantom{=}} \norm{\dot{z}(\tau,\vdot)}_{\Lebesgue[\rho]{2}(\Omega,\R^n)}^2 +
		                         \sum_{K=1}^{N}\alpha_K\kappa_K^{[1]}\norm{\Jacobian z(\tau,\vdot)}_{\Lebesgue{2}(\Omega,\R^{n\times d})}^2\nonumber\\
		&\leq                  \norm{u_2-\tilde{u}_2}_{\Lebesgue[\rho]{2}(\Omega,\R^n)}^2 +
		                         \sum_{K=1}^{N}\alpha_K\mu_K^{[1]}
														 \norm{\Jacobian u_1-\Jacobian\tilde{u}_1}_{\Lebesgue{2}(\Omega,\R^{n\times d})}^2 \mathop{+}\nonumber\\
		&\mathrel{\phantom{=}} \mathop{+} 2\int_{0}^{\tau}\norm{(\dot{f}-\dot{\tilde{f}})(t,\vdot)}_{\Lebesgue{2}(\Omega,\R^n)} \mathop{\times}\nonumber\\
		&\mathrel{\phantom{=}} \mathop{\times} \left\{\norm{\dot{z}(t,\vdot)}_{\Lebesgue[\rho]{2}(\Omega,\R^n)}^2 +
		                         \sum_{K=1}^{N}\alpha_K\kappa_K^{[1]}
														 \norm{\Jacobian z(t,\vdot)}_{\Lebesgue{2}(\Omega,\R^{n\times n})}^2\right\}^{1/2}\dint t \mathop{+}\nonumber\\
		&\mathrel{\phantom{=}} \mathop{+} \int_{0}^{\tau}\Bigr[\inner{2\dot{z}(t,\vdot)}{\Sigma_1(t,\vdot)}_{\Lebesgue{2}(\Omega,\R^n)} +
		                         \inner{2\dot{z}(t,\vdot)}{\Sigma_2(t,\vdot)}_{\Lebesgue{2}(\Omega,\R^n)} + I_2(t)\Bigr]\dint t\text{.}\label{PreEstimateForz}
	\end{align}
    The next step of the proof is to find appropriate estimates of the three terms in square brackets of expression \eqref{PreEstimateForz}. By means of the Cauchy-Schwarz inequality
    we deduce
    \begin{equation*}
	       \inner{2\dot{z}(t,\vdot)}{\Sigma_1(t,\vdot)}_{\Lebesgue{2}(\Omega,\R^n)}
		\leq 2\norm{\dot{z}(t,\vdot)}_{\Lebesgue{2}(\Omega,\R^n)}\left(\int_{\Omega}\norm{\Sigma_1(t,\vdot)}_2^2\dint x\right)^{1/2}\text{,}
	\end{equation*}
	so that we proceed by estimating $\norm{\Sigma_1(t,x)}_2^2$.
	Straightforward calculations show that
	\begin{align*}
	  \norm{\Sigma_1(t,x)}_2 &\leq                  \sum_{k=1}^{n}\sum_{\ell=1}^{d}\sum_{i=1}^{n}\sum_{j=1}^{d}\sum_{K=1}^{N}
		                                                \alpha_K\bigl|\partial_\ell\partial_{Y_{ij}}\partial_{Y_{k\ell}}
																										C_K\bigl(x,\Jacobian u(t,x)\bigr) \mathop{+}\\
		                       &\mathrel{\phantom{=}} \mathop{-} \partial_\ell\partial_{Y_{ij}}\partial_{Y_{k\ell}}
													                          C_K\bigl(x,\Jacobian\tilde{u}(t,x)\bigr)\bigr| \cdot \abs{\partial_j\tilde{w}_i(t,x)} \mathop{+}\\
													 &\mathrel{\phantom{=}} \mathop{+} \sum_{k=1}^{n}\sum_{\ell=1}^{d}\sum_{i=1}^{n}\sum_{j=1}^{d}\sum_{K=1}^{N}
													                          \alpha_K\Bigl|\bigl[\partial_{Y_{ij}}\partial_{Y_{k\ell}}C_K\bigl(x,\Jacobian u(t,x)\bigr) \mathop{+}\\
													 &\mathrel{\phantom{=}} \mathop{-} \partial_{Y_{ij}}\partial_{Y_{k\ell}}
													                          C_K\bigl(x,\Jacobian\tilde{u}(t,x)\bigr)\bigr]\partial_\ell\partial_j\tilde{w}_i(t,x)\Bigr|\\
													 &\mathrel{\phantom{=}} \mathop{+}\sum_{k=1}^{n}\sum_{\ell=1}^{d}\sum_{i=1}^{n}\sum_{j=1}^{d}\sum_{K=1}^{N}
													                          \alpha_K\Biggl|\sum_{p=1}^{n}\sum_{q=1}^{d}\partial_{Y_{pq}}\partial_{Y_{ij}}\partial_{Y_{k\ell}}
																										C_K\bigl(x,\Jacobian u(t,x)\bigr)\partial_\ell\partial_qu_p(t,x) \mathop{+}\\
													 &\mathrel{\phantom{=}} \mathop{-} \partial_{Y_{pq}}\partial_{Y_{ij}}\partial_{Y_{k\ell}}
													                          C_K\bigl(x,\Jacobian\tilde{u}(t,x)\bigr)\partial_\ell\partial_q\tilde{u}_p(t,x)\Biggr| \cdot
													                          \abs{\partial_j\tilde{w}_i(t,x)}\\
													 &=:                     S_{1}(t,x) + S_2(t,x) + S_3(t,x)
	\end{align*}
	in which $S_1$, $S_2$ and $S_3$ are abbreviations for the three fivefold sums appearing in the expression above in the very same order.
	According to the a priori estimate (\ref{APriori316}) and inequality (\ref{Bound310}), the first of these terms satisfies
	\begin{align*}
	  S_1(t,x) &\leq                  2M_1\sum_{k=1}^{n}\sum_{\ell=1}^{d}\sum_{i=1}^{n}\sum_{j=1}^{d}\sum_{K=1}^{N}\alpha_K \mathop{\times}\\
		         &\mathrel{\phantom{=}}\hspace*{-1.8cm} \mathop{\times} \int_{0}^{1}
						                          \Biggl|\sum_{p=1}^{n}\sum_{q=1}^{d}\partial_{Y_{pq}}\partial_\ell\partial_{Y_{ij}}\partial_{Y_{k\ell}}
						                          C_K\bigl(x,s\Jacobian u(t,x)+(1-s)\Jacobian\tilde{u}(t,x)\bigr)
             \bigl(\partial_qu_p(t,x)-\partial_q\tilde{u}_p(t,x)\bigr)\Biggr|\dint s\\
						 &\hspace*{-1.5cm}\leq                  2M_1\sum_{k=1}^{n}\sum_{\ell=1}^{d}\sum_{i=1}^{n}\sum_{j=1}^{d}\sum_{K=1}^{N}\alpha_K \mathop{\times}\\
		         &\mathrel{\phantom{=}}\hspace*{-1.8cm} \mathop{\times} \int_{0}^{1}
						                          \left[\sum_{p=1}^{n}\sum_{q=1}^{d}\abs{\partial_{Y_{pq}}\partial_\ell\partial_{Y_{ij}}\partial_{Y_{k\ell}}
						                          C_K\bigl(x,s\Jacobian u(t,x)+(1-s)\Jacobian\tilde{u}(t,x)\bigr)}^2\right]^{1/2} \cdot \norm{\Jacobian v}_{\Frobenius}\dint s\\
						 &\hspace*{-1.5cm}\leq                  2M_1(nd)^{5/2}\sum_{K=1}^{N}\alpha_K\mu_K^{[7]}\norm{\Jacobian v(t,x)}_{\Frobenius}\text{,}
	\end{align*}
	where we additionally used the Cauchy-Schwarz inequality.\\
	An analog reasoning for $S_2 (t,x)$ with using the a priori estimate (\ref{APriori316}) instead of
	(\ref{APriori317}) as well as inequality (\ref{Bound310}) instead of (\ref{Bound35}) yields
	\begin{equation*}
	  S_2(t,x) \leq 2M_2(nd)^{5/2}\sum_{K=1}^{N}\alpha_K\mu_K^{[2]}\norm{\Jacobian v(t,x)}_{\Frobenius}\text{.}
	\end{equation*}
	The third term $S_3 (t,x)$ is treated in the same way, this time applying both a priori estimates (\ref{APriori316}) and (\ref{APriori317}),
	as well as inequalities (\ref{Bound35}) and (\ref{Bound36}) along with the Cauchy-Schwarz inequality. In this way we obtain
	\begin{align*}
	  S_3(t,x) &\leq                  2M_1\sum_{K=1}^{N}\alpha_K\sum_{k=1}^{n}\sum_{\ell=1}^{d}\sum_{i=1}^{n}\sum_{j=1}^{d}\sum_{p=1}^{n}\sum_{q=1}^{d}
		                                  \abs[auto]{\partial_{Y_{pq}}\partial_{Y_{ij}}\partial_{Y_{k\ell}}C_K\bigl(x,\Jacobian u(t,x)\bigr)} \mathop{\times}\\
						 &\mathrel{\phantom{=}} \mathop{\times} \abs{\partial_\ell\partial_q v_p(t,x)} +
						                          \sum_{K=1}^{N}\alpha_K\sum_{k=1}^{n}\sum_{\ell=1}^{d}\sum_{i=1}^{n}\sum_{j=1}^{d}2M_1M_3 \mathop{\times}\\
						 &\mathrel{\phantom{=}} \mathop{\times} \sum_{p=1}^{n}\sum_{q=1}^{d}\int_{0}^{1}\Biggl|\sum_{a=1}^{n}\sum_{b=1}^{d}
						                          \partial_{Y_{ab}}\partial_{Y_{pq}}\partial_{Y_{ij}}\partial_{Y_{k\ell}}
						                          C_K\bigl(x,s\Jacobian u(t,x)+(1-s)\Jacobian\tilde{u}(t,x)\bigr) \mathop{\times}\\
						 &\mathrel{\phantom{=}} \mathop{\times} \bigl(\partial_bu_a(t,x)-\partial_b\tilde{u}_a(t,x)\bigr)\Biggr|\dint s\\
						 &\leq                  2M_1n^{5/2}d^2\sum_{K=1}^{N}\alpha_K\mu_K^{[2]}\left[\sum_{p=1}^{n}\sum_{|\beta|\leq 2}
						                          \abs{\partial_{\beta}v_p(t,x)}^2\right]^{1/2} \mathop{+}\\
						 &\mathrel{\phantom{=}} \mathop{+} 2M_1M_3(nd)^{7/2}\sum_{K=1}^{n}\alpha_K\mu_K^{[3]}\norm{\Jacobian v(t,x)}_{\Frobenius}\text{.}
	\end{align*}
	Summarizing these results we may deduce
	\begin{align*}
	  &\mathrel{\phantom{=}} \inner{2\dot{z}(t,\vdot)}{\Sigma_1(t,\vdot)}_{\Lebesgue{2}(\Omega,\R^n)}\\
		&\leq                  2\Biggl\{2\sqrt{2}(nd)^{5/2}\sum_{K=1}^{N}\alpha_K\bigl(M_1\mu_K^{[7]}+M_2\mu_K^{[2]}+M_1M_3\mu_K^{[3]}\bigr)
		\norm{\Jacobian v(t,\vdot)}_{\Lebesgue{2}(\Omega,\R^n)}\\
        &\mathrel{\phantom{=}} +
		                         2\sqrt{2}M_1n^{5/2}d^2\sum_{K=1}^{N}\alpha_K\mu_K^{[2]}
		\left[\sum_{p=1}^{n}\sum_{|\beta|\leq 2}\int_{\Omega}
		                         \abs{\partial_{\beta}v_p(t,x)}^2\dint x\right]^{1/2}\Biggr\}\norm{\dot{z}(t,\vdot)}_{\Lebesgue{2}(\Omega,\R^n)}\\
		&\leq                  2\rho_{\mathrm{min}}^{-1}\Biggl\{2\Biggl\{2\sqrt{2}(nd)^{5/2}\sum_{K=1}^{N}
		                         \alpha_K\bigl(M_1\mu_K^{[7]}+M_2\mu_K^{[2]}+M_1M_3\mu_K^{[3]}\bigr)
		\norm{\Jacobian v(t,\vdot)}_{\Lebesgue{2}(\Omega,\R^n)}\\
        &\mathrel{\phantom{=}} +
		                         2\sqrt{2}M_1n^{5/2}d^2\sum_{K=1}^{N}\alpha_K\mu_K^{[2]}\norm{v(t,\vdot)}_{\Sobolev{2}(\Omega,\R^n)}\Biggr\} \mathop{\times}\\
		&\mathrel{\phantom{=}} \mathop{\times} \left\{\norm{\dot{z}(t,\vdot)}_{\Lebesgue[\rho]{2}(\Omega,\R^n)}^2 +
		                         \sum_{K=1}^{N}\alpha_K\kappa_K^{[1]}\norm{\Jacobian z(t,\vdot)}_{\Lebesgue{2}(\Omega,\R^n)}^2\right\}^{1/2}\text{,}
	\end{align*}
	which represents an upper bound for the first term in \eqref{PreEstimateForz}.\\
	We deal with the second one in a similar way. Applying the inequalities (\ref{EstimateDDCK}),
	(\ref{APriori316}) and (\ref{APriori317}) shows
	\begin{align*}
	  &\mathrel{\phantom{=}} \norm{\Sigma_2(t,\vdot)}_{\Lebesgue{2}(\Omega,\R^n)}^2\\
		&\leq                  \norm{\alpha-\tilde{\alpha}}_{\Sup}^2\int_{\Omega}
		                         \Biggl(\sum_{k=1}^{n}\sum_{\ell=1}^{d}\sum_{i=1}^{n}\sum_{j=1}^{d}\sum_{K=1}^{N}
		                         \abs[auto]{\partial_\ell\bigl[\partial_{Y_{ij}}\partial_{Y_{k\ell}}C_K\bigl(x,\Jacobian\tilde{u}(t,x)\bigr)\bigr]} \cdot
														 \abs{\partial_j\tilde{w}_i(t,x)} \mathop{+}\\
		&\mathrel{\phantom{=}} \mathop{+} \sum_{k=1}^{n}\sum_{\ell=1}^{d}\sum_{i=1}^{n}\sum_{j=1}^{d}\sum_{K=1}^{N}
		                         \abs[auto]{\partial_{Y_{ij}}\partial_{Y_{k\ell}}C_K\bigl(x,\Jacobian\tilde{u}(t,x)\bigr)} \cdot
														 \abs{\partial_\ell\partial_j\tilde{w}_i(t,x)}^2\Biggr)^2\dint x\\
		&\leq                  \norm{\alpha-\tilde{\alpha}}_{\Sup}^2\int_{\Omega}
		                         \Biggl(2M_1\sum_{k=1}^{n}\sum_{\ell=1}{d}\sum_{i=1}^{n}\sum_{j=1}^{d}\sum_{K=1}^{N}
		                         \abs[auto]{\partial_\ell\partial_{Y_{ij}}\partial_{Y_{k\ell}}C_K\bigl(x,\Jacobian\tilde{u}(t,x)\bigr)} \mathop{+}\\
		&\mathrel{\phantom{=}} \mathop{+} 2M_1\sum_{k=1}^{n}\sum_{\ell=1}{d}\sum_{i=1}^{n}\sum_{j=1}^{d}\sum_{K=1}^{N}\sum_{p=1}^{n}\sum_{q=1}^{d}
		                         \abs[auto]{\partial_{Y_{pq}}\partial_{Y_{ij}}\partial_{Y_{k\ell}}C_K\bigl(x,\Jacobian\tilde{u}(t,x)\bigr)} \cdot
														 \abs{\partial_\ell\partial_q\tilde{u}_p(t,x)} \mathop{+}\\
		&\mathrel{\phantom{=}} \mathop{+} 2M_2(nd)^2\sum_{K=1}^{N}\mu_K^{[1]}\Biggr)^2\text{.}
	\end{align*}
	A subsequent usage of inequalities (\ref{Bound35}), (\ref{Bound39}) and the a priori estimates (\ref{APriori316}) produces
	\begin{align*}
	  &\mathrel{\phantom{=}} \norm{\Sigma_2(t,\vdot)}_{\Lebesgue{2}(\Omega,\R^n)}^2\\
		&\leq                  \norm{\alpha-\tilde{\alpha}}_{\Sup}^2\int_{\Omega}\Biggl(2(nd)^2\sum_{K=1}^{N}\left(M_1\mu_K^{[6]}+
		                         M_2\mu_K^{[1]}\right) \mathop{+}\\
		&\mathrel{\phantom{=}} \mathop{+} 2M_1n^2d\sum_{K=1}^{N}\mu_K^{[2]}\sum_{\ell=1}^{n}\sum_{p=1}^{n}\sum_{q=1}^{d}
		                         \abs{\partial_\ell\partial_q\tilde{u}_p(t,x)}\Biggr)^2\dint x\\
		&\leq                  \norm{\alpha-\tilde{\alpha}}_{\Sup}\bigg[2\Big(2(nd)^2\sum_{K=1}^{N}\left(M_1\mu_K^{[6]}+
		                         M_2\mu_K^{[1]}\right)\Big)^2\vol(\Omega)+2\Big(2M_1n^2d\sum_{K=1}^{N}\mu_K^{[2]}\Big)^2nd^4M_0^2\bigg]\text{.}
	\end{align*}
	Thus we may deduce that the second term in \eqref{PreEstimateForz} satisfies
	\begin{align*}
	  &\mathrel{\phantom{=}} \inner{2\dot{z}(t,\vdot)}{\Sigma_2(t,\vdot)}_{\Lebesgue{2}(\Omega,\R^n)}\\
		&\leq                  2\rho_{\mathrm{min}}^{-1}\norm{\alpha-\tilde{\alpha}}_{\Sup}\Biggl[2\sqrt{2\vol(\Omega)}(nd)^2\sum_{K=1}^{N}
		                         \left(M_1\mu_K^{[6]}+M_2\mu_K^{[1]}\right) +
		2\sqrt{2}M_0M_1n^{5/2}d^3\sum_{K=1}^{N}\mu_K^{[2]}\Biggr] \mathop{\times}\\
		&\mathrel{\phantom{=}} \mathop{\times} \left\{\norm{\dot{z}(t,\vdot)}_{\Lebesgue[\rho]{2}(\Omega,\R^n)}^2 +
		                         \sum_{K=1}^{N}\alpha_K\kappa_K^{[1]}\norm{\Jacobian z(t,\vdot)}_{\Lebesgue{2}(\Omega,\R^n)}^2\right\}^{1/2}\text{.}
	\end{align*}
	It remains to bound the third term in \eqref{PreEstimateForz}, which is $I_2 (t)$. We resort to inequality~\eqref{Bound35} and a priori estimate~\eqref{APriori316} as well as the Cauchy-Schwarz inequality, which ensures us that
	\begin{align*}
	  I_2(t) &\leq                  \sum_{k=1}^{n}\sum_{\ell=1}^{d}\sum_{i=1}^{n}\sum_{j=1}^{d}\sum_{K=1}^{N}\alpha_K
		                                \int_{\Omega}\sum_{p=1}^{n}\sum_{q=1}^{d}
		                                \abs[auto]{\partial_{Y_{pq}}\partial_{Y_{ij}}\partial_{Y_{k\ell}}C_K\bigl(x,\Jacobian u(t,x)\bigr)} \mathop{\times}\\
					 &\mathrel{\phantom{=}} \mathop{\times} \abs{\partial_q\dot{u}_p(t,x)} \cdot \abs{\partial_jz_i(t,x)} \cdot
					                          \abs{\partial_{\ell}z_k(t,x)}\dint x\\
					 &\leq                  ndM_1\sum_{K=1}^{N}\alpha_K\mu_K^{[2]}\int_{\Omega}nd\sum_{k=1}^{n}\sum_{\ell=1}^{d}
					                          \abs{\partial_\ell z_k(t,x)}^2\dint x\\
					 &\leq                  (nd)^2M_1\left(\sum_{K=1}^{N}\alpha_K\kappa_K^{[1]}\right)^{-1}\sum_{K=1}^{N}\alpha_K\mu_K^{[2]} \mathop{\times}\\
					 &\mathrel{\phantom{=}} \mathop{\times} \left\{\norm{\dot{z}(t,\vdot)}_{\Lebesgue[\rho]{2}(\Omega,\R^n)}^2 +
					                          \sum_{K=1}^{N}\alpha_K\kappa_K^{[1]}\norm{\Jacobian z(t,\vdot)}_{\Lebesgue{2}(\Omega,\R^n)}^2\right\}^{1/2}\text{.}
	\end{align*}
	Substituting the estimates of the three terms into \eqref{PreEstimateForz} finally gives~\eqref{EstimateForz}.\\[1ex]
	
	In the remaining part of the proof we attract our attention to the estimate~\eqref{EstimateForIVz} of the initial value of $z$.
	We start by applying the Cauchy-Schwarz inequality to \eqref{InitialValuesz}, which yields
	\begin{align*}
	  &\mathrel{\phantom{=}} \norm{u_2-\tilde{u}_2}_{\Lebesgue[\rho]{2}(\Omega,\R^n)}\\
		&\leq                  \sqrt{3}\eigennorm[auto]{\sum_{K=1}^{N}(\alpha_K-\tilde{\alpha}_K)
		                         \div\bigl[\nablabf_YC_K\bigl(\vdot,\Jacobian\tilde{u}_0(\vdot)\bigr)\bigr]}_{\Lebesgue{2}(\Omega,\R^n)} \mathop{+}\\
		&\mathrel{\phantom{=}} \mathop{+} \sqrt{3}\eigennorm[auto]{\sum_{K=1}^{N}\alpha_k\div\bigl[\nablabf_YC_K\bigl(\vdot,\Jacobian u_0(\vdot)\bigr)-
		                         \nablabf_YC_K\bigl(\vdot,\Jacobian\tilde{u}_0(\vdot)\bigr)\bigr]}_{\Lebesgue{2}(\Omega,\R^n)} \mathop{+}\\
		&\mathrel{\phantom{=}} \mathop{+} \sqrt{3}\rho_{\mathrm{max}}\norm{(f-\tilde{f})(0,\vdot)}_{\Lebesgue{2}(\Omega,\R^n)}\text{.}
	\end{align*}
	Compared to (\ref{EstimateForIVz}) we only need to take care of the first two summands.
	For the first one, we apply inequality~\eqref{EstimateDDCK} and the Cauchy-Schwarz inequality to get
	\begin{align*}
	  &\mathrel{\phantom{=}} \eigennorm[auto]{\sum_{K=1}^{N}(\alpha_K-\tilde{\alpha}_K)
		                         \div\bigl[\nablabf_YC_K\bigl(x,\Jacobian\tilde{u}_0(x)\bigr)\bigr]}_2^2\\
		&\leq                  n\left\{d\sum_{K=1}^{N}\abs{\alpha_K-\tilde{\alpha}_K}\mu_K^{[4]} +
		                         \sum_{K=1}^{N}\abs{\alpha_K-\tilde{\alpha}_K}\mu_K^{[1]}\sum_{\ell=1}^{d}\sum_{i=1}^{n}\sum_{j=1}^{d}
														 \abs{\partial_\ell\partial_j\tilde{u}_{0,i}(x)}\right\}^2\\
		&\leq                  2n\left(d\sum_{K=1}^{N}\abs{\alpha_K-\tilde{\alpha}_K}\mu_K^{[4]}\right)^2 +
	  2(nd)^2\left(\sum_{K=1}^{N}\abs{\alpha_K-\tilde{\alpha}_K}\mu_K^{[1]}\right)^2\sum_{i=1}^{n}
		                         \sum_{|\beta|\leq 2}\abs{\partial_{\beta}\tilde{u}_{0,i}}^2\text{.}
	\end{align*}
	Integration over $\Omega$ yields
	\begin{align*}
	  &\mathrel{\phantom{=}} \eigennorm[auto]{\sum_{K=1}^{N}(\alpha_K-\tilde{\alpha}_K)
		                         \div\bigl[\nablabf_YC_K\bigl(\vdot,\Jacobian\tilde{u}_0(\vdot)\bigr)\bigr]}_{\Lebesgue{2}(\Omega,\R^n)}\\
		&\leq                  \sum_{K=1}^{N}\abs{\alpha_K-\tilde{\alpha}_K}\left(\sqrt{2\vol(\Omega)}n^{1/2}d\mu_K^{[4]}+
		                         \sqrt{2}nd\mu_K^{[1]}\norm{\tilde{u}_0}_{\Sobolev{2}(\Omega,\R^n)}\right)\text{.}
	\end{align*}
	The estimation of the second term starts with an application of a priori estimate~\eqref{APriori317} and inequality~\eqref{EstimateDDCK},
	which leads to
	\begin{align*}
	  &\mathrel{\phantom{=}} \eigennorm[auto]{\sum_{K=1}^{N}\alpha_K\div\bigl[\nablabf_YC_K\bigl(x,\Jacobian u_0(x)\bigr)-
		                         \nablabf_YC_K\bigl(x,\Jacobian\tilde{u}_0(x)\bigr)\bigr]}_2^2\\
		&\leq                  \sum_{k=1}^{n}\Biggl(\sum_{K=1}^{N}\alpha_K\sum_{\ell=1}^{d}
		                         \abs{\partial_\ell\partial_{Y_{k\ell}}C_K\bigl(x,\Jacobian u_0(x)\bigr)-
														 \partial_\ell\partial_{Y_{k\ell}}C_K\bigl(x,\Jacobian\tilde{u}_0(x)\bigr)} \mathop{+}\\
		&\mathrel{\phantom{=}} \mathop{+} \sum_{K=1}^{N}\alpha_K\sum_{\ell=1}^{d}\sum_{i=1}^{n}\sum_{j=1}^{d}\mu_K^{[1]}
		                         \abs{\partial_\ell\partial_j(u_{0,i}-\tilde{u}_{0,i})(x)} \mathop{+}\\
		&\mathrel{\phantom{=}} \mathop{+} \sum_{K=1}^{N}\alpha_K\sum_{\ell=1}^{d}\sum_{i=1}^{n}\sum_{j=1}^{d}
		                         \abs{\partial_{Y_{ij}}\partial_{Y_{k\ell}}C_K\bigl(x,\Jacobian u_0(x)\bigr)-\partial_{Y_{ij}}\partial_{Y_{k\ell}}
														 C_K\bigl(x,\Jacobian\tilde{u}_0(x)\bigr)}\cdot M_3\Biggr)^2\\
		&\leq                  \sum_{k=1}^{n}\Biggl(\sum_{K=1}^{N}\alpha_K\sum_{\ell=1}^{d}\int_{0}^{1}\sum_{i=1}^{n}\sum_{j=1}^{d}
		                         \abs{\partial_{Y_{ij}}\partial_\ell\partial_{Y_{k\ell}}
														 C_K\bigl(x,s\Jacobian u_0(x)+(1-s)\Jacobian\tilde{u}_0(x)\bigr)} \mathop{\times}\\
		&\mathrel{\phantom{=}} \mathop{\times} \abs{\partial_i(u_{0,i}-\tilde{u}_{0,i})(x)}\dint s + \sum_{K=1}^{N}\alpha_K\mu_K^{[1]}
		                         \sum_{\ell=1}^{d}\sum_{i=1}^{n}\sum_{j=1}^{d}
		                         \abs{\partial_\ell\partial_j(u_{0,i}-\tilde{u}_{0,i})(x)} \mathop{+}\\
		&\mathrel{\phantom{=}} \mathop{+} M_3\sum_{K=1}^{N}\alpha_K\sum_{\ell=1}^{d}\sum_{i=1}^{n}\sum_{j=1}^{d}\int_{0}^{1}\sum_{p=1}^{n}\sum_{q=1}^{d}
		                         \abs{\partial_{Y_{pq}}\partial_{Y_{ij}}\partial_{Y_{k\ell}}
														 C_K\bigl(x,s\Jacobian u_0(x)+(1-s)\Jacobian\tilde{u}_0(x)\bigr)} \mathop{\times}\\
		&\mathrel{\phantom{=}} \mathop{\times} \abs{\partial_q(u_{0,p}-\tilde{u}_{0,p})(x)}\dint s\Biggr)^2\\
		&\leq                  \Biggl[3n^2d\left(d\sum_{K=1}^{N}\alpha_K\mu_K^{[5]}\right)^2 +
		                         3n^2d^2\left(\sum_{K=1}^{N}\alpha_K\mu_K^{[1]}\right)^2 \mathop{+}\\
		&\mathrel{\phantom{=}} \mathop{+} 3n^2d\left(nd^2M_3\sum_{K=1}^{N}\alpha_K^{[2]}\right)^2\Biggr]\sum_{k=1}^{n}
		                         \sum_{|\beta|\leq 2}\abs{\partial_{\beta}(u_{0,k}-\tilde{u}_{0,k})(x)}^2\text{,}
	\end{align*}
	where we used inequalities~\eqref{Bound35}, \eqref{Bound38} and the Cauchy-Schwarz inequality to prove the last step.
	Integration over $\Omega$ finally gives 
	\begin{align*}
	  &\mathrel{\phantom{=}} \eigennorm[auto]{\sum_{K=1}^{N}\alpha_K\div\bigl[\nablabf_YC_K\bigl(\vdot,\Jacobian u_0(\vdot)\bigr)-
		                         \nablabf_YC_K\bigl(\vdot,\Jacobian\tilde{u}_0(\vdot)\bigr)\bigr]}_{\Lebesgue{2}(\Omega,\R^n)}\\
		&\leq                  \sqrt{3}nd\sum_{K=1}^{N}\alpha_K\left(d^{1/2}\mu_K^{[5]}+\mu_K^{[1]}+
		                         nd^{3/2}M_3\mu_K^{[2]}\right)\norm{u_0-\tilde{u}_0}_{\Sobolev{2}(\Omega,\R^n)}
	\end{align*}
	and inequality~\eqref{EstimateForIVz} can now easily be verified. This completes the proof.
\end{proof}

We have now all ingredients together to prove the important norm bound for $z$. This is subsumed in the following theorem.\\

\begin{theorem}\label{T-upper-bound-z}
  There are functions $C_0,\ldots,C_5:[0,\infty[^N\to[0,\infty[$, such that
	\begin{align*}
	  &\mathrel{\phantom{=}} \norm{\dot{z}(\tau,\vdot)}_{\Lebesgue[\rho]{2}(\Omega,\R^n)}^2 +
		                         \sum_{K=1}^{N}\alpha_K\kappa_K^{[1]}\norm{\Jacobian z(\tau,\vdot)}_{\Lebesgue{2}(\Omega,\R^n)}^2\\
		&\leq                  \Biggl\{C_0(\alpha)\norm{u_0-\tilde{u}_0}_{\Sobolev{2}(\Omega,\R^n)} +
		                         C_1(\alpha)\norm{\Jacobian u_1-\Jacobian\tilde{u}_1}_{\Lebesgue{2}(\Omega,\R^n)} +
														 C_2(\alpha)\norm{\alpha-\tilde{\alpha}}_{\Sup} \mathop{+}\\
		&\mathrel{\phantom{=}} \mathop{+} C_3(\alpha)\norm{\Jacobian v}_{\Lebesgue{\infty}((0,T),\Lebesgue{2}(\Omega,\R^{n\times d}))} +
		                         C_4(\alpha)\int_{0}^{\tau}\norm{v(t,\vdot)}_{\Sobolev{2}(\Omega,\R^n)}\dint t \mathop{+}\\
		&\mathrel{\phantom{=}} \mathop{+} C_5(\alpha)\norm{f-\tilde{f}}_{\SobolevW{1}{1}((0,T),\Lebesgue{2}(\Omega,\R^n))}\Biggr\}^2\text{.}
	\end{align*}
	The functions $C_0,\ldots,C_5$ are explicitly given by
	\begin{align*}
	  C_0(\alpha) &:=                     3nd\exp\bigl((nd)^2M_1\overline{C}(\alpha)T/2\bigr)
		                                      \sum_{K=1}^{N}\alpha_K\left(\mu_K^{[1]}+nd^{3/2}M_3\mu_K1{[2]}+d^{1/2}\mu_K^{[5]}\right)\text{,}\\
		C_1(\alpha) &:=                     \exp\bigl((nd)^2M_1\overline{C}(\alpha)T/2\bigr)\left(\sum_{K=1}^{N}\alpha_K\mu_K^[2]\right)^{1/2}\text{,}\\
		C_2(\alpha) &:=                     \exp\bigl((nd)^2M_1\overline{C}(\alpha)T/2\bigr) \mathop{\times}\\
		            &\mathrel{\phantom{:=}} \mathop{\times} \sum_{K=1}^{N}\left(\sqrt{6}nd\mu_K^{[1]}\norm{\tilde{u}_0}_{\Sobolev{2}(\Omega,\R^n)} +
								                          \sqrt{6\vol(\Omega)}n^{1/2}d\mu_K^{[4]}\right) \mathop{+}\\
							  &\mathrel{\phantom{:=}} \mathop{+} 2\sqrt{2}(nd)^2\rho_{\mathrm{min}}^{-1}\frac{\exp\bigl((nd)^2M_1\overline{C}(\alpha)T/2\bigr)-1}
								                          {(nd)^2M_1\overline{C}(\alpha)/2} \mathop{\times}\\
								&\mathrel{\phantom{:=}} \mathop{\times} \sum_{K=1}^{N}\left(\sqrt{\vol(\Omega)}M_2\mu_K^{[1]}+n^{1/2}dM_0M_1\mu_K^{[2]}+
								                          \sqrt{\vol(\Omega)}M_1\mu_K^{[6]}\right)\text{,}\\
		C_3(\alpha) &:=                     2\sqrt{2}(nd)^{5/2}\rho_{\mathrm{min}}^{-1}\frac{\exp\bigl((nd)^2M_1\overline{C}(\alpha)T/2\bigr)-1}
								                          {(nd)^2M_1\overline{C}(\alpha)/2} \mathop{\times}\\
								&\mathrel{\phantom{:=}} \mathop{\times} \sum_{K=1}^{N}\alpha_K\left(M_2\mu_K^{[2]}+ndM_1M_3\mu_K^{[3]}+M_1\mu_K^{[7]}\right)\text{,}
								\end{align*}
    \begin{align*}
		C_4(\alpha) &:=                     2\sqrt{2}n^{5/2}d^2M_1\exp\bigl((nd)^2M_1\overline{C}(\alpha)T/2)\rho_{\mathrm{min}}^{-1}
		                                      \sum_{K=1}^{N}\alpha_K\mu_K^{[2]}\text{,}\\
		C_5(\alpha) &:=                     2\sqrt{3}\exp\bigl((nd)^2M_1\overline{C}(\alpha)T/2\bigr)\rho_{\mathrm{max}}
		                                      \max\{\overline{\overline{C}},1\} \mathop{+}\\
								&\mathrel{\phantom{:=}} \mathop{+} 2\frac{\exp\bigl((nd)^2M_1\overline{C}(\alpha)T/2\bigr)-1}
								                          {(nd)^2M_1\overline{C}(\alpha)/2}\max\{\overline{\overline{C}},1\}\text{,}
	\end{align*}
	where $\overline{\overline{C}}$ is an upper bound for the norm of the embedding
	\begin{equation*}
	  \iota: \SobolevW{1}{1}\bigl((0,T),\Lebesgue{2}(\Omega,\R^n)\bigr)\to\Lebesgue{\infty}\bigl((0,T),\Lebesgue{2}(\Omega,\R^n)\bigr)\text{.}
	\end{equation*}
\end{theorem}

\begin{proof}
The proof follows from combining an application of Gronwall's lemma to inequality~\eqref{EstimateForz} with estimate~\eqref{EstimateForIVz} and some straightforward calculations.\\
\end{proof}


\subsection{An upper bound for the $\Sobolev{2}$-norm of $v(\tau,\cdot)$}

Our next aim is to find an upper bound for the $\Sobolev{2}(\Omega,\R^n)$-norm of $v(\tau,\cdot)$ for fixed $\tau\in [0,T]$.

Some straightforward transformations of \eqref{PDEv} give
\begin{align}
  &\mathrel{\phantom{=}} \rho(x)\dot{z}(t,x) - \rho(x)(f-\tilde{f})(t,x)\nonumber\\
	&=                     \sum_{k=1}^{n}\sum_{\ell=1}^{d}\sum_{i=1}^{n}\sum_{j=1}^{d}\sum_{p=1}^{n}\sum_{q=1}^{d}\sum_{K=1}^{N}\alpha_K
	                         \mathop{\times}\nonumber\\
	&\mathrel{\phantom{=}} \mathop{\times} \int_{0}^{1}\partial_{Y_{pq}}\partial_{Y_{k\ell}}\partial_{Y_{ij}}
	                         C_K\bigl(x,s\Jacobian u(t,x)+(1-s)\Jacobian\tilde{u}(t,x)\bigr)\dint s 
	 e_k\partial_qv_p(t,x)\partial_\ell\partial_ju_i(t,x) \nonumber\\
	&\mathrel{\phantom{=}} \mathop{+} \sum_{k=1}^{n}\sum_{\ell=1}^{d}\sum_{i=1}^{n}\sum_{j=1}^{d}\sum_{K=1}^{N}\alpha_k\partial_{Y_{ij}}\partial_{Y_{k\ell}}
	                         C_K\bigl(x,\Jacobian\tilde{u}(t,x)\bigr)e_k\partial_\ell\partial_jv_i(t,x) \nonumber\\
	&\mathrel{\phantom{=}} \mathop{+} \sum_{k=1}^{n}\sum_{\ell=1}^{d}\sum_{i=1}^{n}\sum_{j=1}^{d}\sum_{K=1}^{N}(\alpha_K-\tilde{\alpha}_K)
	                         \partial_{Y_{ij}}\partial_{Y_{k\ell}}
													 C_K\bigl(x,\Jacobian\tilde{u}(t,x)\bigr)e_k\partial_\ell\partial_j\tilde{u}_i(t,x)\nonumber\\
	&\mathrel{\phantom{=}} \mathop{+} \sum_{k=1}^{n}\sum_{\ell=1}^{d}\sum_{i=1}^{n}\sum_{j=1}^{d}\sum_{K=1}^{N}\alpha_K
	                         \int_{0}^{1}\partial_{Y_{ij}}\partial_{Y_{k\ell}}
													 C_K\bigl(x,s\Jacobian u(t,x)+(1-s)\Jacobian\tilde{u}(t,x)\bigr)\dint s\,\, e_k\partial_jv_i(t,x)\nonumber\\
	&\mathrel{\phantom{=}} +
	                         \sum_{k=1}^{n}\sum_{\ell=1}^{d}\sum_{K=1}^{N}(\alpha_K-\tilde{\alpha}_K)e_k\partial_\ell\partial_{Y_{k\ell}}
													 C_K\bigl(x,\Jacobian\tilde{u}(t,x)\bigr)\text{,}\label{SecondPDEv}
\end{align}
which is a partial differential equation for $v$.
This point of view is the key idea to prove the following theorem.\\

\begin{theorem}\label{theorm329}
  Adopt the notations made before. Then we have
	\begin{align*}
	  &\mathrel{\phantom{=}} \norm{v(\tau,\vdot)}_{\Sobolev{2}(\Omega,\R^n)}\\
		&\leq                  \exp\bigl(\overline{C}(\alpha)C_4(\alpha)\tau\bigr)\hat{C}(\alpha)
		                         \Biggl\{C_0(\alpha)\norm{u_0-\tilde{u}_0}_{\Sobolev{2}(\Omega,\R^n)} \\
		&\mathrel{\phantom{=}} \mathop{+} C_1(\alpha)\norm{\Jacobian u_1-\Jacobian\tilde{u}_1}_{\Lebesgue{2}(\Omega,\R^{n\times d})} \\
		&\mathrel{\phantom{=}} \mathop{+} \left[C_2(\alpha) + n^{1/2}d\sum_{K=1}^{N}\left(n^{1/2}d^2\mu_K^{[1]}+\sqrt{\vol(\Omega)}\mu_K^{[4]}\right)\right]
		                         \norm{\alpha-\tilde{\alpha}}_{\Sup} \\
		&\mathrel{\phantom{=}} \mathop{+} \left[C_3(\alpha) + nd^{1/2}\sum_{K=1}^{N}\alpha_K\left(ndM_3\mu_K^{[2]}+\mu_K^{[5]}\right)\right]
		                         \norm{\Jacobian v}_{\Lebesgue{\infty}((0,T),\Lebesgue{2}(\Omega,\R^n))} \\
		&\mathrel{\phantom{=}} \mathop{+} [C_5(\alpha)+\rho_{\mathrm{max}}]\norm{f-\tilde{f}}_{\Sobolev{1}{1}((0,T),\Lebesgue{2}(\Omega,\R^n))}\Biggr\}\text{.}
	\end{align*}
\end{theorem}

\begin{proof}
  An application of Theorem \ref{theorem:Cordes} to (\ref{SecondPDEv}) implies $v\in\Lebesgue{\infty}\bigl((0,T),\Sobolev{2}(\Omega,\R^n)\bigr)$ as well as
	\begin{align*}
	  &\mathrel{\phantom{=}} \norm{v(\tau,\vdot)}_{\Sobolev{2}(\Omega,\R^n)}\\
		&\leq                  \hat{C}(\alpha)\Biggl\{\norm{\dot{z}(\tau,\vdot)}_{\Lebesgue[\rho]{2}(\Omega,\R^n)} +
		                         \rho_{\mathrm{max}}\norm{(f-\tilde{f})(\tau,\vdot)}_{\Lebesgue{2}(\Omega,\R^n)} \\
		&\mathrel{\phantom{=}} \mathop{+} \Biggl\|\sum_{k=1}^{n}\sum_{\ell=1}^{d}\sum_{i=1}^{n}\sum_{j=1}^{d}\sum_{p=1}^{n}\sum_{q=1}^{d}\sum_{K=1}^{N}
		                         \alpha_Ke_k \mathop{\times}\\
		&\mathrel{\phantom{=}} \mathop{\times} \int_{0}^{1}\partial_{Y_{pq}}\partial_{Y_{ij}}\partial_{Y_{k\ell}}
		                         C_K\bigl(\vdot,s\Jacobian u(\tau,\vdot)+(1-s)\Jacobian\tilde{u}(\tau,\vdot)\bigr)\dint s
		\,\, \partial_qv_p(\tau,\vdot)\partial_\ell\partial_ju_i(\tau,\vdot)\Biggr\|_{\Lebesgue{2}(\Omega,\R^n)} \\
		&\mathrel{\phantom{=}} \mathop{+} \eigennorm[auto]{\sum_{k=1}^{n}\sum_{\ell=1}^{d}\sum_{i=1}^{n}\sum_{j=1}^{d}\sum_{K=1}^{N}
		                         (\alpha_K-\tilde{\alpha}_K)e_k\partial_{Y_{ij}}\partial_{Y_{k\ell}}
														 C_K\bigl(\vdot,\Jacobian\tilde{u}(\tau,\vdot)\bigr)
														 \partial_\ell\partial_j\tilde{u}_i(\tau,\vdot)}_{\Lebesgue{2}(\Omega,\R^n)} \\
		&\mathrel{\phantom{=}} \mathop{+} \Biggl\|\sum_{k=1}^{n}\sum_{\ell=1}^{d}\sum_{i=1}^{n}\sum_{j=1}^{d}\sum_{K=1}^{N}\alpha_Ke_k
		                         \int_{0}^{1}\partial_{Y_{ij}}\partial_\ell\partial_{Y_{k\ell}}
														 C_K\bigl(\vdot,s\Jacobian u(\tau,\vdot)+(1-s)\Jacobian\tilde{u}(\tau,\vdot)\bigr)\dint s 
		\,\, \partial_jv_i(\tau,\vdot)\Biggr\|_{\Lebesgue{2}(\Omega,\R^n)} \\
		&\mathrel{\phantom{=}} \mathop{+} \eigennorm[auto]{\sum_{k=1}^{n}\sum_{\ell=1}^{d}\sum_{K=1}^{N}(\alpha_K-\tilde{\alpha}_K)e_k
		                         \partial_\ell\partial_{Y_{k\ell}}C_K\bigl(\vdot,\Jacobian\tilde{u}(\tau,\vdot)\bigr)}_{\Lebesgue{2}(\Omega,\R^n)}\Biggr\}\text{.}
	\end{align*}
	To finish the proof we have to find upper bounds for each of the norms on the right-hand side of the latter estimate.
	
	Applying premise~\eqref{Bound35} and the a priori estimate~\eqref{APriori317}, we see that
	\begin{align*}
	  &\mathrel{\phantom{=}} \Biggl\|\sum_{k=1}^{n}\sum_{\ell=1}^{d}\sum_{i=1}^{n}\sum_{j=1}^{d}\sum_{p=1}^{n}\sum_{q=1}^{d}\sum_{K=1}^{N}
		                         \alpha_Ke_k \mathop{\times}\\
		&\mathrel{\phantom{=}} \mathop{\times} \int_{0}^{1}\partial_{Y_{pq}}\partial_{Y_{ij}}\partial_{Y_{k\ell}}
		                         C_K\bigl(\vdot,s\Jacobian u(\tau,\vdot)+(1-s)\Jacobian\tilde{u}(\tau,\vdot)\bigr)\dint s 
		\,\, \partial_qv_p(\tau,\vdot)\partial_\ell\partial_ju_i(\tau,\vdot)\Biggr\|_{\Lebesgue{2}(\Omega,\R^n)}\\
		&\leq                  n^2d^{5/2}M_3\sum_{K=1}^{N}\alpha_K\mu_K^{[2]}\norm{\Jacobian v(\tau,\vdot)}_{\Lebesgue{2}(\Omega,\R^{n\times d})}\text{,}
	\end{align*}
	and premise~\eqref{EstimateDDCK} as well as the a priori estimate~\eqref{APriori316} ensure us that we have
	\begin{align*}
    &\mathrel{\phantom{=}} \mathop{+} \eigennorm[auto]{\sum_{k=1}^{n}\sum_{\ell=1}^{d}\sum_{i=1}^{n}\sum_{j=1}^{d}\sum_{K=1}^{N}
		                         (\alpha_K-\tilde{\alpha}_K)e_k\partial_{Y_{ij}}\partial_{Y_{k\ell}}
														 C_K\bigl(\vdot,\Jacobian\tilde{u}(\tau,\vdot)\bigr)
														 \partial_\ell\partial_j\tilde{u}_i(\tau,\vdot)}_{\Lebesgue{2}(\Omega,\R^n)}\\
		&\leq                  nd^3M_0\sum_{K=1}^{N}\mu_K^{[1]}\norm{\alpha-\tilde{\alpha}}_{\Sup}\text{.}
	\end{align*}
	The usage of inequality \eqref{Bound38} leads to
	\begin{align*}
	  &\mathrel{\phantom{=}} \Biggl\|\sum_{k=1}^{n}\sum_{\ell=1}^{d}\sum_{i=1}^{n}\sum_{j=1}^{d}\sum_{K=1}^{N}\alpha_Ke_k
		                         \int_{0}^{1}\partial_{Y_{ij}}\partial_\ell\partial_{Y_{k\ell}}
														 C_K\bigl(\vdot,s\Jacobian u(\tau,\vdot)+(1-s)\Jacobian\tilde{u}(\tau,\vdot)\bigr)\dint s 
		\,\,\partial_jv_i(\tau,\vdot)\Biggr\|_{\Lebesgue{2}(\Omega,\R^n)}\\
		&\leq                  nd^{3/2}\sum_{K=1}^{N}\alpha_K\mu_K^{[5]}\norm{\Jacobian v(\tau,\vdot)}_{\Lebesgue{2}(\Omega,\R^{n\times d})}
	\end{align*}
	and an application of estimate~\eqref{Bound37} finally gives us
	\begin{align*}
		&\mathrel{\phantom{=}} \eigennorm[auto]{\sum_{k=1}^{n}\sum_{\ell=1}^{d}\sum_{K=1}^{N}(\alpha_K-\tilde{\alpha}_K)e_k
		                         \partial_\ell\partial_{Y_{k\ell}}C_K\bigl(\vdot,\Jacobian\tilde{u}(\tau,\vdot)\bigr)}_{\Lebesgue{2}(\Omega,\R^n)}\\
		&\leq                  \sqrt{\vol(\Omega)}n^{1/2}d\sum_{K=1}^{N}\mu_K^{[4]}\norm{\alpha-\tilde{\alpha}}_{\Sup}\text{.}
	\end{align*}
\end{proof}


\subsection{Finishing the proof of Theorem \ref{MainTheorem}}

Let
\begin{align*}
  E(\alpha) &:= \exp\bigl((nd)^2M_1\overline{C}(\alpha)T/2\bigr)\\
	F(\alpha) &:= \frac{\exp\bigl((nd)^2M_1\overline{C}(\alpha)T/2\bigr)}{(nd)^2M_1\overline{C}(\alpha)/2}\text{.}
\end{align*}
From Theorem \ref{theorem319} we infer
\begin{align*}
  &\mathrel{\phantom{=}} \left[\norm{\dot{v}(\tau,\vdot)}_{\Lebesgue[\rho]{2}(\Omega,\R^n)}^2 +
	                         \sum_{K=1}^{N}\alpha_K\kappa_K^{[1]}\norm{\Jacobian v(\tau,\vdot)}_{\Lebesgue{2}(\Omega,\R^{n\times d})}^2\right]^{1/2}\\
	&\leq                  E(\alpha)\left[\norm{u_1-\tilde{u}_1}_{\Lebesgue[\rho]{2}(\Omega,\R^n)}^2 + \sum_{K=1}^{N}\alpha_K\mu_K^{[1]}
											     \norm{\Jacobian u_0-\Jacobian\tilde{u}_0}_{\Lebesgue{2}(\Omega,\R^{n\times d})}^2\right]^{1/2} \mathop{+}\\
	&\mathrel{\phantom{=}} \mathop{+} F(\alpha)\max\{\overline{\overline{C}},1\}
	                         \norm{f-\tilde{f}}_{\SobolevW{1}{1}((0,T),\Lebesgue{2}(\Omega,\R^n))} \mathop{+}\\
	&\mathrel{\phantom{=}} \mathop{+} F(\alpha)\rho_{\mathrm{min}}^{-1}\sum_{K=1}^{N}
	                         \left(\sqrt{\vol(\Omega)}n^{1/2}d\mu_K^{[4]}+nd^2M_0\mu_K^{[1]}\right)\norm{\alpha-\tilde{\alpha}}_{\Sup}\text{.}
\end{align*}
Combining this with Theorems \ref{T-upper-bound-z} and \ref{theorm329} finally yields the estimate \eqref{MainEstimate} and completes the proof of Theorem \ref{MainTheorem}.


\section{Conclusions}

We presented an existence, uniqueness and stability result for a nonlinear initial-boundary value problem that in the special case of $n=d=3$ describes the elastic behavior of a certain class of hyperelastic materials. Particularly we assumed that the stored energy function $C(x,Y)$ may vary in space and is represented by a conic combination of finitely many given fucntions $C_K (x,Y)$. So far we did not give any criterion if this is a reasonable assumption. This subject is postponed to a subsequent article. Further premises are $\Omega$ to have a $\mathcal{C}^2$-boundary and that $C_K$ are in $\mathcal{C}^4$. The result of Theorem \ref{MainTheorem} is of great interest especially when investigating inverse problems such as the reconstruction of the energy function from given measurements. Numerical solvers for such problems usually demand for solutions of the forward problem which then might be described by the IBVP \eqref{PDE}--\eqref{InitialValues}.


\section*{Acknowledgments}
The authors thank Dr. Frank Sch\"opfer and Dr. Frank Binder for many helpful discussions.

\bibliographystyle{siam}
\bibliography{references-hyperelastic} 

\end{document}